\documentclass[11pt,reqno]{amsart}

\usepackage[bottom=3.5cm]{geometry}
\usepackage{pstricks}
\usepackage{float, graphicx}
\usepackage{latexsym}
\usepackage{mathrsfs}
\usepackage{amsmath, amsthm, amssymb,amsfonts}
\usepackage{epsfig}
\usepackage{verbatim}
\usepackage{url}
\usepackage[colorlinks, bookmarks=true]{hyperref}
\usepackage{cleveref}
\usepackage{graphicx}
\usepackage{enumerate}
\usepackage{bm}
\usepackage{dsfont}
\usepackage{stmaryrd}
\usepackage{enumitem}
\usepackage{mathtools}
\usepackage{comment}
\usepackage{color}
\usepackage{xcolor}
\usepackage[font={scriptsize}]{caption}  

\newtheorem{theorem}{Theorem}[section]
\newtheorem{lemma}[theorem]{Lemma}
\newtheorem{definition}[theorem]{Definition}
\newtheorem{proposition}[theorem]{Proposition}
\newtheorem{corollary}[theorem]{Corollary}

\newtheorem{claim}[theorem]{Claim}

\theoremstyle{definition}   
\newtheorem{remark}[theorem]{Remark}

\newcommand{\beq}{\begin{equation}}
	\newcommand{\eeq}{\end{equation}}

\usepackage{pgf,tikz}
\usetikzlibrary{arrows}
\usetikzlibrary{arrows.meta}
\usetikzlibrary{calc}
\usetikzlibrary{shapes}
\usetikzlibrary{backgrounds}
\usetikzlibrary{decorations.markings}
\usetikzlibrary{decorations}
\usetikzlibrary{decorations.pathreplacing}
\tikzset{individu/.style={fill,thick,circle}}

\def\r{{\mathbb R}}
\def\e{{\mathbb E}}
\def\p{{\mathbb P}}

\def\P{{\bf P}}

\def\z{{\mathbb Z}}
\def\N{{\mathbb N}}
\def\T{{\mathbb T}}

\def\ee{e}
\def\d{\, \mathrm{d}}

\def\ej{{\mathbf e}}

\newcommand{\dd}{\mathrm d}

\newcommand{\Radi}{\mathsf{rad}_{\operatorname{int}}} 
\newcommand{\Rade}{\mathsf{rad}_{\operatorname{ext}}} 
\newcommand{\clt}{\mathcal{C}} 
\newcommand{\aut}{\operatorname{Aut}} 
\newcommand{\dist}{\mathrm{dist}}
\newcommand{\critical}{\mathrm{c}}
\newcommand{\wt}{\widetilde}

\newcommand{\tclt}{\widetilde{\clt}} 

\newcommand{\1}{{\text{\Large $\mathfrak 1$}}}

\allowdisplaybreaks

\numberwithin{equation}{section}
\usepackage{environ}

\begin{document}
	\makeatletter
	\def\@settitle{\begin{center}%
			\baselineskip14\p@\relax
			\normalfont\LARGE
			\@title
		\end{center}%
	}
	\makeatother
	\title[Speed of Random walk on dynamical percolation in nonamenable Graphs]{ \bf Speed of random walk on dynamical percolation \\ in nonamenable transitive graphs}
	\let\MakeUppercase\relax 
	
	\author[C. Gu]{\large \bf Chenlin Gu} 
	\address{Yau Mathematical Sciences Center, Tsinghua University, Beijing, China.}
	\email{gclmath@tsinghua.edu.cn}
	
	\author[J. Jiang]{Jianping Jiang} 
	\address{Yau Mathematical Sciences Center, Tsinghua University, Beijing, China.}
	\email{jianpingjiang@tsinghua.edu.cn}

	\author[Y. Peres]{Yuval Peres}
	\address{Beijing Institute of Mathematical Sciences and Applications, Beijing, China.}
	\email{yperes@gmail.com}
	
	\author[Z. Shi]{Zhan Shi}
	\address{Academy of Mathematics and Systems Science, Chinese Academy of Sciences, Beijing, China.}
	\email{shizhan@amss.ac.cn}
	
	\author[H. Wu]{Hao Wu}
	\address{Yau Mathematical Sciences Center, Tsinghua University, and Beijing Institute of Mathematical Sciences and Applications, Beijing, China.}
	\email{hao.wu.proba@gmail.com}
	
	\author[F. Yang]{Fan Yang}
	\address{Yau Mathematical Sciences Center, Tsinghua University, and Beijing Institute of Mathematical Sciences and Applications, Beijing, China.}
	\email{fyangmath@tsinghua.edu.cn}

	\maketitle

	\begin{abstract}
		Let $G$ be a nonamenable transitive unimodular graph. In dynamical percolation, every edge in $G$ refreshes its status at rate $\mu>0$, and following the refresh, each edge is open independently with probability $p$. The random walk traverses $G$ only along open edges, moving at rate $1$. In the critical regime $p=p_c$, we prove that the speed of the random walk is at most $O(\sqrt{\mu \log(1/\mu)})$, provided that  $\mu \le e^{-1}$. In the supercritical regime $p>p_c$, we prove that the speed on $G$ is of order 1 (uniformly in $\mu)$, while in the subcritical regime $p<p_c$, the speed is of order $\mu\wedge 1$. 
		
		
	\end{abstract}
	
	
	\section{Introduction}
	\label{s:intro}
	
	Let $G=(V, E)$ be an infinite graph, where $V$ and $E$ denote the sets of vertices and edges, respectively. The \textbf{dynamical percolation process} on $G$ is defined as follows: the edges in $E$ are either open or closed and independently refresh their status at an exponential rate $\mu>0$. Upon refreshing, each edge becomes open with probability $p\in (0, \, 1)$ or closed with probability $(1-p)$. (It is important to note the distinction between ``refreshing" and ``flipping": when an edge is refreshed, it has probability $p$ of becoming open, regardless of its previous state.) Fixing a starting point $X_0=o\in V$, we initiate a continuous-time random walk denoted as $(X_t)_{t \geq 0}$ on $G$. The particle jumps with rate $1$ and randomly selects one of its neighboring vertices to jump to, provided that the connecting edge is open at that specific time.
	
	We study the random walk on dynamical percolation in nonamenable graphs. Specifically, we are interested in the interplay between the refreshing rate $\mu$ and the speed of random walk. Recall that the \textbf{Cheeger constant} of a graph $G$ is defined by 
	\begin{align}\label{eq.defCheeger}
		\Phi(G) := \inf\left\{\frac{|\partial_E W|}{\sum_{v \in W} \deg(v)}: W \text{ finite subset of } V\right\},
	\end{align}
	where $\partial_{E} W$ denotes the \textbf{edge boundary} of $W$, i.e., the set of edges in $E$ with one endpoint in $W$ and the other endpoint in $V \setminus W$. The graph $G$ is said to be \textbf{nonamenable} if $\Phi(G) > 0$; otherwise it is \textbf{amenable}, i.e., $\Phi(G) = 0$. 
	
	When $G$ is \textbf{transitive}, the random walk $(X_t)_{t \geq 0}$ exhibits a well-defined speed, i.e., given any $\mu>0$ and $p\in (0,1)$, there exists a constant $v_p(\mu) \in [0, 1]$ such that
	\begin{equation}\label{eq.speed}
		\lim_{t\to \infty} \frac{\mathrm{dist}(X_0, \, X_t)}{t} = v_p(\mu) \quad\hbox{\rm a.s.\ and in $L^1$} .  
	\end{equation}
	Here, $\mathrm{dist}(X_0, \, X_t)$ represents the graph distance on $G$ between $X_0$ and $X_t$. 
	(For the proof of this fact, see \Cref{lem:existence_speed} below.) We are particularly interested in the dependence of $v_p(\mu) $ on $\mu$ when $\mu$ is much smaller than 1. This corresponds to a scenario where the status of individual edges   changes   more slowly  than the random walk moves. 
	Note that if $\mu\to \infty$, then $(X_t)_{t\ge 0}$ converges weakly to the simple random walk on $G$ with time scaled by $p$. Similarly, if $\mu$ is of order 1, we may expect that the system behaves in various ways like the ordinary random walk. This is also justified by our theoretical results in Theorems \ref{thm:Subcritical} and \ref{thm:Supercritical} below. 
	
	When $\mu$ is small, the random walk on dynamical percolation can also be viewed as an approximation of the random walk on static percolation. In particular, similar to the random walk on static percolation, the behavior of the random walk on dynamical percolation also exhibits a phase transition as $p$ crosses the critical point.
	More precisely, let $\P_p$ denote the probability measure of the (static) Bernoulli-$p$ bond percolation on $G$, where every edge is independently open with probability $p \in [0,1]$. 
	The connected components of vertices with respect to the open edges in the percolation are called {\bf clusters}, and recall that \textbf{the critical point of connectivity} $p_c = p_c(G)$ is defined by
	\begin{align}\label{eq.defpc}
		p_c(G) := \sup\left\{ p\in [0,1]: \text{ every cluster has finite size } \P_p\text{-a.s.}\right\}.
	\end{align}
	For infinite transitive graphs with bounded degree and superlinear volume growth, we know that $0 < p_c < 1$; see \cite[Theorem~1.3]{DGRSY2020}. 
	Heuristically, in the subcritical phase $p <p_c$, the moving particle is confined within a finite cluster and must wait for a duration of order $1/\mu$ before the cluster undergoes any notable changes. This suggests that the speed in the subcritical phase should be of order $O(\mu)$. 
	On the other hand, in the supercritical phase $p > p_c$, there is a positive probability for the moving particle to be inside an infinite cluster. In such cases, the speed of the random walk should be of order $O(1)$ and is not affected much by the evolution of the graph.    
	Now, we state the first two results on the speed of the random walk in the subcritical and supercritical phases, respectively. One can notice the qualitative difference between them. 

	\begin{theorem}[Subcritical phase]\label{thm:Subcritical}
		Let $G$ be a connected, locally finite, nonamenable transitive unimodular graph where each vertex has degree $d \ge 3$. For any $p\in (0,p_c)$, there exist two constants $0 < c_{\ref{thm:Subcritical}} < C_{\ref{thm:Subcritical}} < \infty$ independent of $\mu$ such that the speed of the random walk satisfies the following estimate: 
		\begin{align}\label{eq:Subcritical}
			c_{\ref{thm:Subcritical}} (\mu \wedge 1) \leq v_p(\mu) \leq \left(C_{\ref{thm:Subcritical}}\mu \right)\wedge 1.
		\end{align}
	\end{theorem}
	
	\begin{theorem}[Supercritical phase]\label{thm:Supercritical}
		Let $G$ be a connected, locally finite, nonamenable transitive unimodular graph where each vertex has degree $d\ge 3$. For any $p\in (p_c, 1]$, there exists a constant $ c_{\ref{thm:Supercritical}} \in (0,1)$ independent of $\mu$ such that the speed of random walk satisfies the following estimate:  
		\begin{align}\label{eq:Supercritical}
			c_{\ref{thm:Supercritical}} \leq v_p(\mu) \leq 1.
		\end{align}
	\end{theorem}
	
	The proofs of Theorems~\ref{thm:Subcritical} and~\ref{thm:Supercritical} will be presented in Sections~\ref{sec_sub} and \ref{sec_super}, respectively. 
	By the definition of the model, the upper bound $v_p(\mu)\le 1$ is trivial for all $\mu>0$ and $p\in [0,1]$. The upper bound $C_{\ref{thm:Subcritical}}\mu$ in \eqref{eq:Subcritical} is proved following our above heuristic reasoning. The main challenges lie in establishing the lower bounds in \eqref{eq:Subcritical} and \eqref{eq:Supercritical} when $\mu\ll 1$. 
	
	Next, we consider the random walk on dynamical percolation at criticality $p=p_c$. In \Cref{t:ub}, we will show that the (order of the) speed is strictly smaller than that in the supercritical regime. 
	The proof depends crucially on the \textbf{one-arm estimate} for static percolation on $G$. Denote by $\clt_v$ the open cluster containing the vertex $v \in V$, and let $\Radi(\clt_v)$ be its intrinsic radius, i.e., the maximum, over vertices $w \in \clt_v$, of the intrinsic graph distance (using open edges) between $v$ and $w$.
	
	\begin{definition}[One-arm estimate] An infinite graph $G$ is said to satisfy the one-arm estimate with exponent 1 if there exists a positive constant $C>0$ such that for every $v \in G$ and $r > 0$,
		\begin{align}\label{eq.OneArm}
			\P_{p_c}\left(\Radi(\clt_v) \geq r\right) \leq \frac{C}{r}.
		\end{align}
	\end{definition}

	\begin{theorem}[Critical phase]
		\label{t:ub}
		Let $G$ be a connected, locally finite, transitive unimodular graph where each vertex has degree $d\ge 3$. Suppose it satisfies the one-arm estimate \eqref{eq.OneArm}. Then, there exists a constant $C_{\ref{t:ub}} \in (0,\infty)$ (independent of $\mu$) such that 
		\begin{equation}\label{eq:boundvpc}
			{v_{p_c}(\mu)} \le C_{\ref{t:ub}} \sqrt{\mu \log(1/\mu)}, \quad  \forall \mu\in(0,1/e] . 
		\end{equation}
	\end{theorem}
	
	While the statement above does not explicitly mention nonamenability, it is hinted in the one-arm estimate \eqref{eq.OneArm}. The one-arm estimate is implied by \textbf{the mean-field behavior} of the critical percolation model. A classical condition to verify the mean-field behavior of critical percolation is the well-known \textbf{triangle condition} introduced by Aizenman and Newman \cite{Triangle1984}: for a transitive graph $G=(V,E)$ with an origin $o \in V$, the triangle function is defined as 
	\begin{align}\label{eq.triangle}
		\nabla_p := \sum_{x,y \in V} \P_{p}(o \leftrightarrow x) \P_{p}(x \leftrightarrow y) \P_{p}(y \leftrightarrow o).
	\end{align}
	Then, the triangle condition is $\nabla_p < \infty$, which implies the one-arm estimate \eqref{eq.OneArm} following the argument by Kozma and Nachmias in \cite{KozmaNachmiasAOConjecture}; see also the discussion by Hutchcroft in \cite[Section~7]{Hutchcroft2020nonunimodular}. Schonmann conjectured in \cite[Conjectures~1.1 and ~1.2]{Schonmann2001} that critical percolation on all nonamenable transitive graphs exhibits mean-field behavior. Recently, Hutchcroft proved this conjecture for all nonamenable nonunimodular graphs in \cite{Hutchcroft2020nonunimodular} and for a large family of nonamenable unimodular graphs in \cite{Hutchcroft2019hyper,Hutchcroft2020L2,Hutchcroft2022slightly}, which includes the Gromov hyperbolic graphs. In addition, another $L^2$-boundedness condition was proposed in \cite{Hutchcroft2020L2} as a sufficient criterion for the triangle condition.


	
	Finally, we also obtain a lower bound estimate for the speed at criticality. The following estimate \eqref{eq:lower_vpmu_general} holds for all $p\in (0,1]$ and provides an optimal lower bound for the subcritical phase. However, it fails to yield a sharp lower bound for both the critical and supercritical phases.
	
	\begin{proposition}\label{prop:general_Lower}
		Let $G$ be a connected, locally finite, nonamenable transitive graph where each vertex has degree $d \ge 3$. For any $p\in (0, 1]$, there exists a constant $c_{\ref{prop:general_Lower}} \in (0, \infty)$ independent of $\mu$ such that the random walker satisfies the following lower bound:
		\begin{align}\label{eq:lower_vpmu_general}
			v_{p}(\mu) \geq c_{\ref{prop:general_Lower}} (\mu\wedge 1).
		\end{align}
	\end{proposition}

	It is natural to ask similar questions about the mean squared displacement $\e[|X_t-X_0|^2]$ (corresponding to our $\e[\dist(X_0,X_t)]$) and the diffusion constant $D_p(\mu)$ (corresponding to our speed $v_p(\mu)$) for the random walk on dynamical percolation in $\z^d$ lattices. Sharp estimates on the mean squared displacement and diffusion constant in the subcritical regime have been derived by Peres, Stauffer, and Steif in \cite{peres-stauffer-steif}. Some estimates on the mixing time and mean squared displacement in the supercritical case were also proved by Peres, Sousi, and Steif in \cite{PSS20}, although the results there do not provide a precise lower bound for the diffusion constant. A general lower bound for the diffusion constant, analogous to \Cref{prop:general_Lower}, for all $p\in (0,1]$ was established by Peres, Sousi, and Steif in \cite{PSS18}. 
	The critical regime for the random walk on dynamical percolation in $\z^d$ lattices, however, has not been addressed in the literature so far. In a companion paper \cite{GJPSWY_Zd}, we investigate the random walk on critical dynamical percolation in $\z^d$ lattices and establish a sharp lower bound for the diffusion constant in the supercritical case. We obtain similar results to those in \Cref{t:ub} and \Cref{thm:Supercritical} for the diffusion constant, which show that the critical and supercritical regimes are different qualitatively. 
	
	An important special case of nonamenable transitive unimodular graphs is the regular tree $\mathbb T_d$, where each vertex has degree $d$. It is a widely accepted belief that the random walk on trees is closely connected to the random walk on critical percolation in high-dimensional $\z^d$ lattices. This belief is due to the observation that critical percolation clusters in high-dimensional lattices exhibit only local cycles and are similar  to critical percolation on a tree. 
	This heuristic has been rigorously justified in various contexts. For instance, Kozma and Nachmias proved the Alexander-Orbach conjecture in high dimensions~\cite{KozmaNachmiasAOConjecture}, which was previously proved for trees by Kesten~\cite{kesten86} and Barlow and Kumagai~\cite{BarlowKumagai2006}. See also the survey by Heydenreich and van der Hofstad~\cite{HeydenreichvanderHofstadsurvey}.
	Our results in this paper, along with those presented in the companion paper \cite{GJPSWY_Zd}, provide further evidence regarding the connection between the diffusion constant on high-dimensional lattices and the speed on trees. In fact, the proofs in these two papers mutually inspire each other, particularly in the critical case.

	\subsection*{Related works} 
	The study of percolation on general graphs beyond $\mathbb{Z}^d$ has received significant attention in recent decades. In their seminal paper \cite{BS1996}, Benjamini and Schramm proposed a systematic investigation of percolation on general quasi-transitive graphs. In particular, they posed many conjectures regarding the existence of phase transitions and the extinction of infinite clusters at criticality.
	(Recall that a graph $G$ is called quasi-transitive if the action of the automorphism group $\aut(G)$ on $V$ has only a finite number of orbits.) Benjamini, Lyons, Peres, and Schramm \cite{BLPS1999B,BLPS1999A} proved that for nonamenable quasi-transitive unimodular graphs, $p_c < 1$ and no infinite cluster exists at criticality. Further details can be found in Chapter 8 of the book by Lyons and Peres \cite{LP16}. 
	The developments in percolation on nonunimodular graphs usually rely on the study of geometric conditions. Lyons \cite{Lyons1995} proved that every Cayley graph with exponential growth exhibits a non-trivial phase transition. Teixeira \cite{Teixeira2016} proved that $p_c < 1$ for graphs with polynomial growth and satisfying the local isoperimetric inequality of dimension greater than $1$. 
	In a recent groundbreaking result, Duminil-Copin, Goswami, Raoufi, Severo, and Yadin \cite{DGRSY2020} established the existence of phase transitions in all quasi-transitive graphs with super-linear volume growth.  Hermon and Hutchcroft \cite{HermonHutchcroft2021} established the exponential decay of the cluster size distribution and the anchored expansion for infinite clusters in supercritical percolation on transitive nonamenable graphs. 

	
	There is a vast literature on random walks in evolving random environments. 
	We begin by highlighting some related works that share the same context as our paper, specifically focusing on random walks on dynamical percolation.
	The concept of dynamical percolation on arbitrary graphs was introduced by H{\"a}ggstr{\"o}m, Peres, and Steif \cite{OYS97}. Subsequently, the model of random walk on dynamical percolation in $\mathbb Z^d$ was introduced by Peres, Stauffer, and Steif \cite{peres-stauffer-steif}. 
	As previously mentioned, in the context of subcritical/supercritical dynamical percolation in $\mathbb Z^d$, several results concerning mean squared displacement, mixing times, and hitting times for the random walk have been proved in  
	\cite{peres-stauffer-steif,PSS18,PSS20}. 
	Later, Hermon and Sousi \cite{HS20} extended the setting to general underlying graphs, establishing a comparison principle between the random walk on dynamical percolation and the simple random walk on the graph regarding hitting and mixing times, spectral gap, and log-Sobolev constant. 
	Sousi and Thomas \cite{ST2020} investigated the random walk on dynamical Erd{\H o}s-R{\'e}nyi graphs (i.e., dynamical percolation on complete graphs) and showed that the mixing time exhibits a cutoff phenomenon.  
	More recently, Andres, Gantert, Schmid, and Sousi \cite{AGSS23} studied the biased random walk on dynamical percolation in $\mathbb Z^d$ and established results such as the law of large numbers, the Einstein relation, and monotonicity with respect to the bias. Lelli and Stauffer \cite{LS2209} studied the mixing time of random walk on a dynamical random cluster model in $\mathbb Z^d$, where edge states evolve according to continuous-time Glauber dynamics.

	Many research interests have also been devoted to studying random walks on dynamical models that differ from dynamical percolation in terms of the evolving mechanism of the underlying graph. For example, Avena, G{\"u}lda{\c s}, van der Hofstad, den Hollander, and Nagy \cite{AGHH18,AGHH2019,SPA2022} studied random walks on dynamical configuration models, where a small fraction of edges is sampled and rewired uniformly at random at each unit of time. Caputo and Quattropani \cite{CQ21} investigated the mixing of random walks on dynamic random digraphs, which undergo full regeneration at independent geometrically distributed random time intervals.
	Figueiredo, Iacobelli, Oliveira, Reed, and  Ribeiro \cite{FIORR21} and Iacobelli, Ribeiro, Valle, and Zuazn{\'a}bar \cite{IRVZ22} considered random walks that build their own trees, wherein the trees evolve with time, randomly, and depending upon the walker. 
	Cai, Sauerwald, and Zanetti \cite{CSZ20} and Figueiredo, Nain, Ribeiro, de Souza e Silva, and Towsley \cite{FNRS11} studied random walks on evolving graph models generated as general edge-Markovian processes. Iacobelli and Figueiredo \cite{IF16} considered random walks moving over dynamic networks in the presence of mutual dependencies---the network influences the walker steps and vice versa.
	Avin, Kouck{\'y}, and Lotker \cite{CMZ2018} and Sauerwald and Zanetti \cite{SZ19} derived bounds for cover, mixing, and hitting times of random walks on dynamically changing graphs, where the set of edges changes in each round under some mild assumptions regarding the evolving mechanism. 
	
	In the preceding discussion, we have focused on highlighting a few representative studies that are particularly relevant to random walks in evolving random environments on nonamenable (or tree-like) graphs. There are also numerous references concerning random walks in evolving random environments on $\mathbb{Z}^d$ lattices, which will be reviewed in the companion paper \cite{GJPSWY_Zd}.

	\subsection*{Organization}
	The rest of the paper is organized as follows. 
	In \Cref{sec:preliminary}, we provide several preliminary results that will be used in the proofs of the main results, including the choice of the initial bond configuration, the existence of speed, and the stationarity of the environment seen by the moving particle. In \Cref{s:ub}, we consider the critical case and present the proof of \Cref{t:ub}. The proofs for the supercritical case (\Cref{thm:Supercritical}) and the subcritical case (\Cref{thm:Subcritical}) are presented in \Cref{sec_super} and \Cref{sec_sub}, respectively. 
	Additionally, in Section \ref{sec_sub}, we establish the general lower bound stated in \Cref{prop:general_Lower}. 
	Finally, some concluding remarks and open questions are stated in \Cref{sec:conclusion}. 
	
	\medskip
	\noindent{\bf Notations.} 
	We will use the sets of natural numbers $\N=\{1,2,3,\ldots\}$ and positive real numbers $\mathbb R_+=\mathbb R\cap (0,\infty)$. For any (real or complex) numbers $a$ and $b$, we will use the notations $a\lesssim b$, $a=O(b)$, or $b=\Omega(a)$ to mean that $|a|\le C|b|$ for a constant $C>0$ that does not depend on $\mu$. 
	
	
	\section{Preliminaries}\label{sec:preliminary}

	This section aims to establish several preliminary results regarding the random walk on dynamical percolation. In some results, the underlying graph can be more general than our setting in the main results. These results include the impact of the initial bond configuration of the dynamical percolation, the estimate of reset times, the existence of speed on transitive graphs, and the stationarity of the environment seen by the moving particle (which is the only result in this section that requires unimodularity of the underlying graph). 
	
	Given a graph $G=(V,E)$, we denote by $\dist(u,v)$ the graph distance between two vertices $u, v \in V$. Let $B(v,r):=\{x\in V:\dist(x,v)\le r\}$ denote the $r$-neighborhood of $v$ on $G$. In the presence of percolation, we use $\dist_{\clt}(u,v)$ to denote the chemical distance, i.e., the length of the shortest path connecting $u$ and $v$ that consists only of the open edges. If $u,v$ belong to different clusters, then $\dist_{\clt}(u,v) = \infty$.  We call $\dist(\cdot, \cdot)$ and $ \dist_{\clt}(
	\cdot, \cdot)$ the \textbf{intrinsic} and \textbf{extrinsic} distances on the percolation, respectively. Let $\clt_v$ represent the cluster containing the vertex $v$. 
	We define the extrinsic radius of $\clt_v$ as
	\begin{align}\label{eq.defRadius}
		\Rade (\clt_v) := \sup\{\dist(v, u) : u \in \clt_v\}.
	\end{align}
	Similarly, we define the intrinsic radius of $\clt_v$ as 
	\begin{align}\label{eq.defRadius_int}
		\Radi (\clt_v) := \sup\{\dist_{\clt}(v, u) : u \in \clt_v\}.
	\end{align}
	
	We now introduce the following notations for the random walk on dynamical percolation in an underlying graph $G=(V,E)$: 
	\begin{itemize}
		\item Let $\eta : E \times \r_+ \to \{0,1\}$ denote the state of edges, where $\eta_t(e) = 1$ (resp.~$0$) indicates the edge $e$ is open (resp.~closed) at time $t$. Each edge refreshes independently at a rate $\mu \in (0,\infty)$, and when the refresh happens, the edge is open with probability $p \in (0,1)$ and closed with probability $(1-p)$. For each $e \in E$, we use $0 \leq \chi^e_1 < \chi^e_2 < \cdots$ to denote the sequence of refresh times. Note $\{\chi^e_j\}_{j \in \N}$ forms a Poisson point process of intensity $\mu$ on $\r_+$. 
		
		\item Let $X : \r_+ \to V$ denote the position of the moving particle, which \textbf{attempts to jump} to one of its neighbors uniformly at random and independently of the dynamics $(\eta_t)_{t \geq 0}$,  with a rate of $1$. A jump is successful if and only if the edge the particle attempts to cross is open. Denote by $\xi_k$ the moment of the $k$-th attempt to jump, and let $\ej_k$ denote the corresponding edge the particle attempts to cross. Note $\{\xi_k\}_{k \in \N}$ forms a Poisson point process of intensity $1$ on $\r_+$.
	\end{itemize}
	
	Suppose the particle starts from a vertex $x \in V$. Given an initial bond configuration $\omega \in \{0,1\}^{E}$, we denote by $\p_{\omega, x}$ the probability measure generated by 
	\begin{align}\label{eq:info}
		\left(\{\xi_k, \ej_k\}_{k \in \N}, \{\chi^e_j\}_{j \in \N, e\in E}, X_0 =x, \eta_0 = \omega\right).
	\end{align}
	The process $(X_t, \eta_t)_{t \geq 0}$ is Markovian under $\p_{\omega, x}$. By default, we set $o \in V$ as the initial position of the particle, and write $\p_{\omega} = \p_{\omega, o}$ for short. We use $\p^{\boldsymbol{\eta}}$ to denote the probability measure conditioned on the whole evolution of the dynamical percolation process $(\eta_t)_{t \geq 0}$.  
	Notice that the product Bernoulli measure  $\pi_p:=\mathrm{Ber}(p)^E$ is the invariant measure for the process $(\eta_t)_{t \geq 0}$. Then, we define the probability measure
	\begin{align}\label{eq.defAnnealed}
		\p := \e_{\pi_p}\p_{\omega} ,
	\end{align}
	which is the \textbf{annealed probability measure} with respect to $\bm{\eta}$ when the initial bond configuration is distributed according to $\pi_p$. 

	\subsection{Choosing the initial environment}
	
	The following proposition shows that to prove the main results, it suffices to choose the initial distribution to be stationary.  
	
	\begin{proposition}\label{prop:stationary_arbitrary}
		Consider the random walk $(X_t)_{t\ge 0}$ on dynamical percolation in a connected and locally finite graph $G=(V,E)$, started at $o \in V$. If an event $D$ in the $\sigma$-field generated by $(X_t)_{t \ge 0}$ satisfies ${\mathbb P}_\eta(D)=1$ for a.e.\ initial environment $\eta \in \{0,1\}^E$ sampled  according to the product Bernoulli measure $\pi_p=\mathrm{Ber}(p)^E$, then ${\mathbb P}_\omega(D)=1$ for {\bf every} initial environment $\omega \in \{0,1\}^E$.
	\end{proposition}
	
	\begin{proof}
		Given $\omega \in \{0,1\}^E$ and an integer $r>1$, let $\omega^r$ denote a random initial environment that agrees with $\omega$ on the ball $B(o,r)$ and is i.i.d.~$\mathrm{Ber}(p)$ outside this ball. Since $\omega^r$ is obtained by conditioning the random variables with law $\mathrm{Ber}(p)^E$ on an event of positive probability, we have that $ {\mathbb P}_{\omega^r}(D)=1$.
		
		We couple the processes started with $\omega$ and $\omega^r$, respectively, so that every edge $e$ is refreshed at the same times $\chi^e_1,\ \chi^e_2,  \ldots$,  in both environments and has the same status after updating, i.e., $\omega_t^r(e)=\omega_t(e)$ for all $t \ge \chi^e_1 $. Furthermore, the sequence of times $\xi_j$ when the particle attempts to jump is the same under both measures. If $\{\ej_j\}$ is the sequence of edges selected by the moving particle under ${\mathbb P}_\omega$, we may select the same edges  $\ej_j$ under ${\mathbb P}_{\omega^r}$, for every index $j$   that satisfies        
		$$\forall i < j, \quad  \{ \chi_1^{\ej_i} < \xi_i \quad \text{or} \quad \ej_i \in B(o,r)\}\,. $$
		Next, consider the events
		$$Q_n:=\{\ej_n \ne \ej_j \ \forall j<n \} \,. $$
		For each $n$, the probability that $\ej_n$ is selected for the first time before it is refreshed is
		$${\mathbb P}_\omega\bigl(Q_n \cap \{\xi_n<\chi_1^{\mathbf e_n}\}\bigr)={\mathbb E}_\omega \bigl( e^{-\mu \xi_n}{\bf 1}_{Q_n}\bigr) \le {\mathbb E}_\omega \bigl( e^{-\mu \xi_n} \bigr)=(1+\mu)^{-n}\,, $$
		where the first equality is obtained by conditioning on $\xi_n$, $Q_n$, and $\ej_n$.
		Define 
		$$Q:=\cup_{n>r}\bigl(Q_n \cap \{\xi_n<\chi_1^{\ej_n}\}\bigr)\,.$$
		Given $\varepsilon>0$, we can choose $r=r(\mu,\varepsilon)$ so that
		$$\p_\omega(Q) \le \sum_{n>r}{\mathbb P}_\omega\bigl(Q_n \cap \{\xi_n<\chi_1^{\ej_n}\}\bigr)\le \sum_{n>r} (1+\mu)^{-n}<\varepsilon\,.$$
		On the event $Q^c$, the edges $\{\ej_j\}_{j \ge 1}$ selected under $\p_{\omega^r}$ are the same as those selected under $\p_{\omega}$, and $\omega_{\xi_j}(\ej_j)=\omega^r_{\xi_j}(\ej_j)$ for all $j \ge 1$. The locations of the particle $(X_t)_{t \ge 0}$ are determined by the variables $\{\xi_j,\, \ej_j,\, \eta_{\xi_j}(\ej_j)\}_{j \ge 1}\,,$ so
		$$\p_{\omega}(D) \ge \p_{\omega^r}(D \cap Q^c) \ge 1-\varepsilon \,. $$
		Since $\varepsilon$ can be arbitrarily small, the result follows.    
	\end{proof}

	\subsection{Uniform upper bound for the reset time}        
	We denote by $A_t$ the \textbf{memory set} of  edges that have not been refreshed after the last attempt of the particle to cross it in $[0,t]$, i.e.,  
	\begin{align}\label{eq.Memory}
		A_t := \left\{ e \in E: \max_{k \in \N}\{\xi_k: \xi_k \leq t, \ej_k = e\} > \max_{j \in \N} \{\chi^e_j: \chi^e_j \leq t\}   \right\}.
	\end{align}
	Here, we adopt the convention that $\max_{k \in \N} \emptyset = 0$. Therefore, an edge $e$ that has never been attempted to be crossed before time $t$ is not included in the memory set $A_t$. We define the \textbf{reset time} as the sequence of moments $(T_k)_{k \in \N}$ at which the memory set returns to the empty state from the non-empty state: $T_0 = 0$ and 
	\begin{align}\label{eq.weakRegeneration}
		\begin{split}
			T_k:= \inf\left\{t > T_{k-1}: \vert A_t\vert = 0 \ \text{ and } \sup_{s \in (T_{k-1},t)} \vert A_s\vert \geq 1 \right\},\quad \forall k \in \N .
		\end{split}
	\end{align}
	The next result gives a rough uniform bound on this reset time. 
	
	\begin{lemma}\label{lem.weakRege}
		For any infinite graph $G$ and any initial bond configuration $\omega \in \{0,1\}^E$, the following upper bound holds for the reset time:
		\begin{align}\label{eq.weakRege}
			\qquad \e_{\omega}[T_k - T_{k-1}] \leq e^{\frac{1}{\mu}}, \quad \forall k \in \N.
		\end{align}
	\end{lemma}
	\begin{proof}
		It suffices to estimate $\e_{\omega}[T_1]$, as the same proof can be applied to $\e_{\omega}[T_k - T_{k-1}]$. We define the hitting time of $1$ for the process $(\vert A_t\vert)_{t \geq 0}$:
		\begin{align*}
			\tau_1 := \inf\left\{t >0: \vert A_t\vert = 1 \right\}.
		\end{align*}
		The proof is based on the observation that the size of the memory set $(\vert A_t\vert)_{t \geq 0}$ is stochastically dominated by a birth-death process $(S_t)_{t \geq 0}$ with birth rate $1$ and death rate $\mu |S_t|$: 
		\begin{align}
			&\vert A_t\vert = S_t, \quad \forall t \in [0, \tau_1], \label{eq.couple1}\\
			&\vert A_t\vert \leq S_t,\quad \forall t \in (\tau_1, \infty).
			\label{eq.couple2}
		\end{align}
		Regarding \eqref{eq.couple1}, since the memory set $A_t$ is empty before $\tau_1$, its increasing rate is the same as that of $S_t$, and we can couple the two processes using the same random variable $\xi_1$. In the subsequent evolution, the increasing rate of the memory set is smaller than $1$ because every attempt to jump may choose an edge already in $A_t$. Meanwhile, the decay rate of the memory set is $\mu \vert A_t \vert$ since every edge refreshes independently with rate $\mu$. This leads to \eqref{eq.couple2}.
		
		
		We define another sequence of reset times $(\widetilde{T}_k)_{k \in \N}$ for the process $(S_t)_{t \geq 0}$. In view of \eqref{eq.couple1} and \eqref{eq.couple2}, we have that 
		\begin{align*}
			\e_{\omega}[T_1] \leq \e[\widetilde{T}_1].
		\end{align*}
		To estimate $\e[\widetilde{T}_1]$, we can calculate the stationary distribution $(\tilde q_n)_{n \in \N}$ for the process $(S_t)_{t \geq 0}$ by solving the detailed balance equation 
		\begin{align}\label{eq.detailedBalance}
			\tilde q_n  = \tilde q_{n+1}\cdot \mu (n+1)  , \qquad \forall n \in \N.
		\end{align} 
		We find that $(\tilde q_n)_{n \in \N}$ follows a $\operatorname{Poisson}(\mu^{-1})$ distribution. Therefore, $(S_t)_{t \geq 0}$ is a positive recurrent Markov process (see e.g., Theorem~21.13 in the book \cite{LP17} by Levin and Peres) and the expectation of  $\widetilde{T}_1$ is given by (see e.g., \cite[Proposition~1.14 (ii)]{LP17})
		\begin{align}\label{eq.BDreturn}
			\e[\widetilde{T}_1] =  \tilde q_0^{-1} = e^{\frac{1}{\mu}}.
		\end{align}
		This concludes \eqref{eq.weakRege}.
	\end{proof}

	\subsection{Existence of speed on transitive graphs}
	
	With \Cref{lem.weakRege}, we can derive that the speed of random walk on dynamical percolation exists for any transitive graph.
	\begin{lemma}\label{lem:existence_speed}
		For any infinite transitive graph $G$ and any initial bond configuration $\omega \in \{0,1\}^E$, the speed $v_p(\mu)$ of $(X_t, \eta_t)_{t \geq 0}$ (recall \eqref{eq.speed}) exists $\p_{\omega}$-a.s.~and in $L^1(\p_\omega)$. Furthermore, $v_p(\mu)$ does not depend on the choice of $\omega$. 
		
	\end{lemma}
	
	\begin{proof}
		For simplicity, we assume the initial configuration $\omega$ is sampled from the product measure $\pi_p$. We will address the general initial configuration using Proposition~\ref{prop:stationary_arbitrary}. 
		
		\medskip
		\noindent \textit{Step~1: limit along the reset times.} Recall the reset times $(T_k)_{k \in \N}$ defined in \eqref{eq.weakRegeneration}. By the strong Markov property, the stationarity of the initial configuration, and the transitivity of $G$, we know that $(T_k - T_{k-1})_{k \in \N}$ are i.i.d.~random variables and $(X_{T_k})_{k \in \N}$ is a discrete random walk on $G$. Since the increments of $\dist(X_0, X_t)$ have a rate at most $1$, $(\dist(X_0, X_t) - t)_{t \geq 0}$ is a supermartingale. Then, the optional stopping time theorem yields that
		\begin{align*}
			\e[\dist(X_0, X_{T_1 \wedge t}) - T_1 \wedge t] \leq 0.
		\end{align*}
		This implies that $\e[\dist(X_0, X_{T_1 \wedge t})] \leq \e[T_1 \wedge t]$. Taking $t \to \infty$, we obtain that
		\begin{align*}
			\e[\dist(X_0, X_{T_1})] &\leq \liminf_{t \to \infty} \e[\dist(X_0, X_{T_1 \wedge t})] \\
			&\le\lim_{t \to \infty }\e[T_1 \wedge t] = \e[T_1] < \infty,
		\end{align*}
		where we applied Fatou's lemma in the first step, the monotone convergence theorem in the third step, and \eqref{eq.weakRege} in the last step.
		Once we have verified that $\e[\dist(X_0, X_{T_1})] < \infty$, according to \cite[Theorem~14.10]{LP17} based on Kingman's subadditive ergodic theorem, the limit
		\begin{align}\label{eq.vSubsequence}
			v' := \lim_{n\to \infty} \frac{\mathrm{dist}(X_0, \, X_{T_n})}{n}
		\end{align}
		exists $\p$-a.s.\ and is a constant.

		\medskip
		\noindent\textit{Step~2: limit along $\r_+$.}
		We define the total number of attempts to jump during $[s, t]$ as
		\begin{align}\label{eq.defJ}
			J[s,t] := \# \{ i \in \N_+: \xi_i \in [s, t]\}.
		\end{align}
		Note that 
		\begin{multline}\label{eq.vSupInf}
			\frac{\mathrm{dist}(X_0, \, X_{T_{n}}) - J[T_{n-1}, T_n]}{T_n} \leq \inf_{t \in (T_{n-1}, T_n]}\frac{\mathrm{dist}(X_0, \, X_t)}{t} \\
			\leq \sup_{t \in (T_{n-1}, T_n]}\frac{\mathrm{dist}(X_0, \, X_t)}{t} \leq \frac{\mathrm{dist}(X_0, \, X_{T_{n-1}}) + J[T_{n-1}, T_n]}{T_{n-1}}.
		\end{multline}
		Hence, it suffices to show that the two sides converge to the same limit as $n \to \infty$. First, we have that
		\begin{align}\label{eq.vSupInf1}
			\lim_{n \to \infty} \frac{\mathrm{dist}(X_0, \, X_{T_{n}})}{T_n} = \lim_{n \to \infty} \frac{\mathrm{dist}(X_0, \, X_{T_{n}})}{n} \cdot \frac{n}{T_n} = \frac{v'}{\e[T_1]}, \qquad \p\text{-a.s.}
		\end{align}
		by \eqref{eq.vSubsequence} and the strong law of large numbers (LLN). Regarding $J[T_{n-1}, T_n]$, notice that $t \mapsto J[0,t]$ is a Poisson point process of intensity $1$, so it admits a strong LLN (see e.g., Theorem~2.4.7 of Durrett \cite{durrett2019probability}):
		\begin{align*}
			\lim_{t \to \infty}\frac{J[0,t]}{t} = 1, \qquad \p\text{-a.s.}
		\end{align*}
		On the other hand, using the LLN again, we have that $T_n/T_{n-1} \to 1$ $\p$-a.s. Hence, we conclude that  
		\begin{align}\label{eq.vSupInf2}
			\lim_{n \to \infty} \frac{J[T_{n-1}, T_n]}{T_{n-1}} = \lim_{n \to \infty} \left[\frac{J[0, T_n]}{T_{n}} \frac{T_n}{T_{n-1}} - \frac{J[0, T_{n-1}]}{T_{n-1}} \right]= 0, \qquad \p\text{-a.s.}
		\end{align}
		Applying \eqref{eq.vSupInf1} and \eqref{eq.vSupInf2} to \eqref{eq.vSupInf} yields that
		\begin{align*}
			v_p(\mu) = \lim_{t\to \infty} \frac{\mathrm{dist}(X_0, \, X_t)}{t} =  \frac{v'}{\e[T_1]}, \qquad \p\text{-a.s.} 
		\end{align*}
		
		For a general initial configuration $\omega$, consider the event $A := \{\lim_{t\to \infty} {\mathrm{dist}(X_0, \, X_t)}/{t} = v_p(\mu)\}$. By \Cref{prop:stationary_arbitrary}, we have that 
		\begin{align*}
			\p_{\omega}(A) = \p(A) = 1,
		\end{align*}
		which implies the $\p_\omega$-a.s.~convergence. Since $v_p(\mu)$ is defined under $\p$, it does not depend on the choice of $\omega$.
		
		\medskip
		\noindent
		\textit{Step~3: $L^1$-convergence.}
		Since $\dist(X_0, X_t) \leq J[0,t]$ and $J[0,t]$ follows a Poisson$(t)$ distribution, for every initial configuration $\omega$, we have that 
		\begin{align}\label{eq.L2UI}
			\e_{\omega}\left[ \left(\frac{\mathrm{dist}(X_0, \, X_t)}{t}\right)^2  \right] &\leq \e_{\omega}\left[ \left(\frac{J[0,t]}{t} \right)^2 \right] = \frac{t^2 + t}{t^2} \le 2, \quad \forall t\ge 1. 
		\end{align}
		Thus, the ratios $({\mathrm{dist}(X_0, \, X_t)}/{t})_{t\ge 1}$ are uniformly integrable under $\p_\omega$. Together with the $\p_\omega$-a.s.~convergence, this justifies the $L^1(\p_\omega)$ convergence.             
	\end{proof}

	\begin{remark}
		We can establish the positivity of the speed $v_p(\mu)$ for discrete nonamenable groups: Kaimanovich and Vershik proved in \cite[Theorem~5]{kaimanovich1983random} that the spectral radius of the random walk $(X_{T_k})_{k \in \N}$ is strictly less than $1$. Proposition 14.6 in \cite{LP16} (which is due to Avez's work in 1976 \cite{Avez1976}) connects this lower bound of the spectral radius to that of entropy, which further gives a lower bound of the speed of the random walk following \cite[Theorem~14.9]{LP16}. \end{remark}

	\subsection{Environment seen by the particle on transitive unimodular graphs}\label{sec:regenerate}
	
	In this subsection, we study the process $(X_t, \eta_t)_{t\ge 0}$ from the view of the particle. Using the reversibility of the process $(X_t, \eta_t)_{t \geq 0}$,  we extend it to a two-sided process $(X_t, \eta_t)_{t \in \mathbb R}$ by first running an independent copy $(X'_t, \eta'_t)_{t \geq 0}$ and then defining
	\begin{align}\label{eq.RWNegative}
		(X_t, \eta_t) := (X'_{-t}, \eta'_{-t}),\quad \forall t < 0.
	\end{align} 
	The random variables in \eqref{eq:info} can also be extended to $(\{\xi_k, \ej_k\}_{k \in \z}, \{\chi^{e}_j\}_{j \in \z, e\in E})$ for the two-sided process $(X_{t}, \eta_{t})_{t \in \r}$. Specifically, for $k<0$, $\xi_{k}$ and $ \ej_{k}$ record the moments and the corresponding edges attempted to be crossed before time $0$, while $\chi^{e}_{k}$ records the moments when the edge $e$ refresh before time $0$.

	To ensure the stationarity of the environment seen by the particle, besides transitivity, we also need to assume that the graph $G$ is unimodular. We now introduce some necessary definitions and notations (for more details, refer to \cite[Chapter~8.2]{LP16}). Let $G=(V,E)$ be a transitive graph and $\aut(G)$ be the group of automorphisms. Then, $G$ is unimodular if its left Haar measure coincides with its right Haar measure. Denote this Haar measure by $H$. Let $\Gamma_{x,y}$ be the set of automorphisms that maps $x$ to $y$:
	\begin{align}\label{eq.defGammaxy}
		\Gamma_{x,y} := \{\gamma \in \aut(G): \gamma x = y\}.
	\end{align}
	Let $H_{o,x}$ represent the normalized probability measure obtained by restricting $H$ to $\Gamma_{x,y}$.

	In this subsection, we always assume that $G=(V,E)$ is a locally finite, transitive, and unimodular graph. We define the environment seen from the particle $(\eta^*_t)_{t \in \r}$ as follows: the process $(\eta^*_t)_{t \in \r}$ takes value in $\{0,1\}^E$ and the particle is always located at the origin $o$. Each edge refreshes independently with rate $\mu$. At an exponential clock of rate $1$, the particle attempts to jump to a neighbor $x$ of the origin $o$, selected uniformly at random,  and the environment is shifted accordingly. 
	More specifically, for the $k$-th attempt to jump at time $\xi_k$, we sample an automorphism $\gamma$ according to the normalized Haar measure $H_{o,x}$. We then define the shift of the environment at $\xi_k$ as follows:
	\begin{align}\label{eq.Shift}
		\eta^*_{\xi_k} = \left\{
		\begin{array}{ll}
			\eta^*_{\xi_k-} \circ \gamma, &\quad \text{ if } \eta^*_{\xi_k-}(\{o,x\}) = 1\\
			\eta^*_{\xi_k-}, &\quad \text{ if }  \eta^*_{\xi_k-}(\{o,x\}) = 0
		\end{array}\right. ,
	\end{align} 
	where $(\eta^*_{\xi_k-} \circ \gamma)(e) := \eta^*_{\xi_k-}(\gamma e)$ for every $e \in E$. Here, in the first case, when the edge $\{o,x\}$ is open, the particle jumps and the corresponding automorphism in $\aut(G)$ applies. The automorphism $\gamma$ is chosen uniformly at random since there can be multiple automorphisms in $\Gamma_{o,x}$. In the second case, when the edge $\{o,x\}$ is closed, the environment remains unchanged and the particle does not move. 
	
	We continue to use $\p$ to denote the probability measure for the extended two-sided process when the initial bond configuration $\eta_0$ is sampled according to $\pi_p$. The next result shows that the process $(\eta^*_t)_{t \in \mathbb R}$ seen from the particle forms a stationary ergodic environment.

	\begin{lemma}\label{lem.ViewStationary}
		Given a transitive unimodular graph $G=(V,E)$ where each vertex has degree $d\ge 3$, the environment $(\eta^*_t)_{t \in \mathbb R}$ seen from the particle is stationary and ergodic under $\p$.
	\end{lemma}
	\begin{proof}
		First, we prove the stationarity. Let $\mathcal{L}$ denote the generator of the dynamics $(\eta^*_t)_{t \in \mathbb R}$: for any bounded measurable function $f : \{0,1\}^E \to \mathbb{R}$, 
		\begin{align*}
			\mathcal{L} f (\eta^*): &= \mathcal{L}_{R} f (\eta^*) + \mathcal{L}_{M} f (\eta^*), 
		\end{align*}
		where $\mathcal{L}_{R}$ and $\mathcal{L}_{M}$ are defined as
		\begin{align*}
			\mathcal{L}_{R} f (\eta^*) &:= \sum_{e \in E} \left[ \mu p \left(f(\eta^{*,e,1}) - f (\eta^*)\right) + \mu(1-p)\left(f(\eta^{*,e,0}) - f (\eta^*)\right)\right],\\
			\mathcal{L}_{M} f (\eta^*) &:= \sum_{x \in V_o} \frac{1}{\vert V_o\vert} \int_{\Gamma_{o,x}}\left[f(\eta^* \circ \gamma) - f (\eta^*)\right]\mathbf{1}_{\{\eta^*(\{o,x\}) = 1\}} \, \d H_{o,x}(\gamma).
		\end{align*}
		Here, $\eta^{*,e,1}$ (resp.~$\eta^{*,e,0}$) denotes the bond configuration obtained by opening (resp.~closing) the edge $e$, and $V_o$ represents the set of neighboring vertices of $o$.
		We have decomposed the generator into two parts $\mathcal{L}_{R}$ and $\mathcal{L}_{M}$, which correspond to the edge refreshing and the movement of the particle, respectively. To establish stationarity, it suffices to prove that 
		$$\e[\mathcal{L} f ] = \e[\mathcal{L}_{R} f ] +\e[\mathcal{L}_{M} f ] = 0.$$ 
		The term $\e[\mathcal{L}_{R} f ]=0$ follows from the fact that $\pi_p$ is an invariant measure for the refreshing generator $\mathcal{L}_{R}$. On the other hand, we have $\e[\mathcal{L}_{M} f ] = 0$ by \cite[Lemma~3.13]{LS99}, where Lyons and Schramm proved that the environment seen by the particle on static percolation is stationary when the graph is transitive unimodular.

		Next, we prove the ergodicity through the following strong mixing condition: for any $T, R \in (0, \infty)$  and cylinder events $A,B \in \sigma(\{0,1\}^{B(o,R)} \times [0,T])$, 
		\begin{align}\label{eq.mixing}
			\lim_{s \to \infty}\p\left(\{\eta^* \in A\} \cap \{T_{s} \eta^* \in B\}\right) = \p\left(\eta^* \in A\right)\p\left(\eta^* \in B\right).
		\end{align}
		Here, the notation $\eta^*$ represents the whole process $(\eta^*_t)_{t \in \mathbb R}$ and $ T_{s} $ is the time-shift operator defined as $(T_{s} \eta^*)_t := \eta^*_{s+t}$. To prove \eqref{eq.mixing}, let $\epsilon > 0$ be a constant that will be chosen later and choose an arbitrary $s > T$. Then, we introduce the  events
		\begin{align*}
			E_1 &:= \{\text{the number of attempted jumps is less than } T+s^{\epsilon} \}, \\
			E_2 &:= \{\text{all the edges in } B(o,R+T+s^{\epsilon}) \text{ refresh at least once during } [T, s] \}.
		\end{align*}
		Clearly, we have that 
		\begin{align}\label{eq.decayE1}
			\lim_{s \to \infty}\p(E_1^c) = 0,
		\end{align}
		since the number of attempted jumps is finite almost surely. On the other hand, using a simple union bound, we obtain that  
		\begin{align}\label{eq.decayE2}
			\limsup_{s \to \infty}\p(E_2^c) \leq \lim_{s\to \infty} d^{R+T+s^{\epsilon}} e^{- \mu (s - T)} = 0,
		\end{align}
		as long as we choose $\epsilon \in (0,1)$ sufficiently small depending on $d$ and $\mu$.  
		
		Denoting ${E_0 := \{\eta^* \in A\} \cap \{T_{s} \eta^* \in B\}}$, we can rewrite the left-hand side (LHS) of \eqref{eq.mixing} using \eqref{eq.decayE1} and \eqref{eq.decayE2} as  
		\begin{align}
			\lim_{s \to \infty}\p\left(E_0\right) &= \lim_{s \to \infty} \p\left(E_0 \cap E_1 \cap E_2\right) + \lim_{s \to \infty}\p\left(E_0 \cap E_1 \cap E_2^c\right) + \lim_{s \to \infty}\p\left(E_0 \cap E_1^c \right) \nonumber \\
			&=\lim_{s \to \infty} \p\left(E_0 \cap E_1 \cap E_2\right). \label{eq.mixingReduce}
		\end{align}
		We decompose the event $E_0 \cap E_1 \cap E_2$ as
		\begin{align}\label{eq.mixingFactor}
			\p\left(E_0 \cap E_1 \cap E_2\right) 
			= \p\left(\{T_{s} \eta^* \in B\} \vert \{\eta^* \in A\} \cap E_1 \cap E_2 \right)\p\left(\{\eta^* \in A\} \cap E_1 \cap E_2\right).
		\end{align}
		Using \eqref{eq.decayE1} and \eqref{eq.decayE2} again, the second factor on the right-hand side (RHS) satisfies
		\begin{align}\label{eq.mixingFactor2}
			\lim_{s \to \infty} \p\left(\{\eta^* \in A\} \cap E_1 \cap E_2\right)  = \p\left(\eta^* \in A\right).
		\end{align}
		For the first factor on the RHS of \eqref{eq.mixingReduce}, the Markov property implies that
		\begin{align}\label{eq.mixingFactor1}
			\p\left(\{T_{s} \eta^* \in B\} \vert \{\eta^* \in A\} \cap E_1 \cap E_2 \right) = \p\left(T_{s} \eta^* \in B \right) = \p\left( \eta^* \in B \right).
		\end{align}
		Here, the first equality holds because every visited edge is refreshed under $E_1 \cap E_2$, and the second equality is due to stationarity. Combining \eqref{eq.mixingReduce}--\eqref{eq.mixingFactor2} concludes \eqref{eq.mixing}.
	\end{proof}
	\begin{remark}
		The ergodicity of the environment seen from the particle will not be used in the rest of the paper. However, it may be of interest for future studies related to random walks on dynamical percolation. 
	\end{remark}

	\section{Speed for the critical case}
	\label{s:ub}
	
	In this section, we prove Theorem~\ref{t:ub}. A key ingredient is the following mean-field behavior exhibited by the supercritical percolation on $G$ with $p$ slightly greater than $p_c$. 
	
	\begin{lemma}\label{lem.SlightSup}
		Given an infinite graph $G$ satisfying the one-arm estimate \eqref{eq.OneArm}, there exist constants $\delta = \delta(G) > 0$ and $C >0$ such
		that for every vertex $v\in V$ and $p \in \left[p_c, p_c + {\delta}/{r} \right]$, 
		\begin{align}\label{eq.SlightSup}
			\P_p\left(\Rade(\clt_v) \geq r\right) \leq \P_p\left(\Radi(\clt_v) \geq r\right) \leq \frac{C}{r},\quad 	\forall r > 0.
		\end{align}
	\end{lemma}
	\begin{proof}
		The first inequality follows immediately from the fact that  $\Rade(\clt_v) \leq \Radi(\clt_v)$. For the second inequality, we introduce the shorthand notation $D_r := \{\Radi(\clt_v) \geq r\}$. Notice that the percolation cluster containing $v$ at $p_c$ can be obtained in two steps: we first construct a Bernoulli-$p$ percolation, and then consider the Bernoulli-$(p_c/p)$ percolation of it. This gives us the inequality
		\begin{align*}
			\P_{p_c}(D_r)
			\ge
			\Big( \frac{p_c}{p} \Big)^{\! r} \P_p(D_r) \ge
			\left(1+\frac{\delta}{p_c r}\right) ^{- r} \P_p(D_r) \ge
			e ^{-\delta/p_c} \P_p(D_r),
		\end{align*}
		where the factor $({p_c}/{p})^r$ above represents the probability that all the $r$ edges in a path realizing $\{\Radi(\clt_v) \geq r\}$ is open, and the second inequality is due to the condition $p \in \left[p_c, p_c + {\delta}/{r} \right]$. Now, applying the one-arm estimate \eqref{eq.OneArm} to $\P_{p_c}(D_r)$, we obtain that
		\begin{align*}
			\P_p(D_r) \leq e ^{\delta/p_c}  \P_{p_c}(D_r) \leq \frac{C}{r},
		\end{align*}
		which completes the proof of the second inequality in \eqref{eq.SlightSup}.
	\end{proof}

	\begin{proof} [Proof of Theorem \ref{t:ub}]
		We study the displacement $\e[\dist(X_0, X_t)]$, where $t=t(\mu)$ will be chosen later. Denote by $\tclt$ the subgraph composed of all the bonds that are open at least once during $[0,t]$. Under the stationary measure $\p$, the subgraph $\tclt$ is a percolation on $G$, where each bond is open with probability 
		\begin{align*}
			p := p_c + (1-p_c)(1-  e^{-\mu t p_c}) \leq p_c(1+\mu t).
		\end{align*}  
		Let $\tclt_o$ denote the cluster containing the root $o$ in $\tclt$, and let $\delta$ be the constant in Lemma~\ref{lem.SlightSup}. Under the condition 
		\begin{align}\label{eq.muSlight}
			0 \leq p_c \mu t \leq {\delta}/{r},
		\end{align}
		we apply Lemma~\ref{lem.SlightSup} to obtain that
		\begin{align}\label{eq.tildeKolmogrov}
			\p\left(\Rade(\tclt_o) \geq r\right) \leq \frac{C}{r}. 
		\end{align}
		
		Next, we analyze the displacement using $\tclt$ and decompose it as follows:
		\begin{equation}
			\begin{split}\label{eq.CriticalDecom}
				\e[\dist(X_0, X_t)] \leq \sum_{k=1}^K 	&~\e\left[\dist(X_0, X_t)\1_{\left\{2^{k-1} \leq  \Rade(\tclt_o) \leq 2^k \right\}}\right] \\
				+ &~\e\left[\dist(X_0, X_t)\1_{\left\{\Rade(\tclt_o) \geq 2^K \right\}}\right],
			\end{split}
		\end{equation}
		where $K$ is a threshold to be determined later. For the scale $2^k$, since the random walk during $[0,t]$ must stay within $\tclt_o$, we have that 
		\begin{align*}
			\e\left[\dist(X_0, X_t)\1_{\left\{2^{k-1} \leq \Rade(\tclt_o) \leq 2^k \right\}}\right] &\leq 2^k \p\left(\Rade(\tclt_o) \geq 2^{k-1}\right)  \leq 2^k \cdot \frac{C}{2^{k-1}} = 2 C, 
		\end{align*}
		where we used the estimate \eqref{eq.tildeKolmogrov} in the second step, assuming that \eqref{eq.muSlight} holds. 
		For the scale beyond the threshold $2^K$, we use the trivial bound $\e[\dist(X_0, X_t) \vert \tclt_o] \le t$ to obtain that
		\begin{align*}
			\e\left[\dist(X_0, X_t)\1_{\left\{\Rade(\tclt_o) \geq 2^K \right\}}\right] &= \e\left[\dist(X_0, X_t) \, \big\vert \Rade(\tclt_o) \geq 2^K \right] \p\left(\Rade(\tclt_o) \geq 2^K \right)\\
			&\leq \frac{C t}{2^{K}}.
		\end{align*}
		Once again, we applied the estimate \eqref{eq.tildeKolmogrov} provided that the condition \eqref{eq.muSlight} holds for $r = 2^K$.  
		We choose $K$ such that 
		$$\frac{\delta}{2^{K+1}}<p_c \mu t\leq \frac{\delta}{2^K}.$$ 
		By choosing this $K$, we ensure that the condition $p_c \mu t \leq {\delta}/{2^k}$ is satisfied in all previous steps for $1 \leq k \leq K$. Plugging the above estimates into \eqref{eq.CriticalDecom}, we obtain that
		\begin{align*}
			\frac{1}{t}\e[\dist(X_0, X_t)] \leq  \frac{2 C}{t} \left(K + \frac{t}{2^{K+1}}\right) \le   2 C \left(\frac{- \log_2(p_c\mu t/\delta) }{t} + \frac{p_c\mu t}{\delta}\right).
		\end{align*} 
		To minimize the RHS, we set $t=t(\mu):= \sqrt{(1/\mu) \log(1/\mu)}$ and obtain that 
		\begin{align}\label{eq.CriticalSmallInterval}
			\frac{1}{t}\e[\dist(X_0, X_t)] \leq  C \sqrt{\mu \log(1/\mu)}.
		\end{align}
		
		Finally, we use the definition \eqref{eq.speed} and the triangle inequality to obtain that 
		\begin{align*}
			v_{p_c}(\mu) &= \lim_{T\to \infty} \frac{\e [\mathrm{dist}(X_0, \, X_T)]}{T}  \\
			&\leq \lim_{T \to \infty} \frac{t}{T} \sum_{n=0}^{\lfloor T/t\rfloor} \frac{1}{t} \e[\dist(X_{nt}, X_{(n+1)t})] \\
			& \leq C \sqrt{\mu \log(1/\mu)}.
		\end{align*} 
		Here, in the second step, we applied the estimate \eqref{eq.CriticalSmallInterval} to every segment $[nt, (n+1)t]$ due to stationarity. This concludes \eqref{eq:boundvpc}.
	\end{proof}

	\section{Speed for the supercritical case}\label{sec_super}

	In this section, we provide the proof of Theorem~\ref{thm:Supercritical}. The upper bound in \eqref{eq:Supercritical} is trivial. 
	When $\mu \ge 1$, the lower bound follows from \Cref{prop:general_Lower}. When $\mu$ is much smaller than 1, we get a substantial improvement over \eqref{eq:lower_vpmu_general}. 
	The proof of the lower bound in this case relies on the Diaconis-Fill coupling~\cite{DF90} between the random walk and the evolving set process (which we will review in Section~\ref{subsec::DFcoupling}). Assuming a key estimate in Lemma~\ref{lem:Phi2}, we complete the proof of Theorem~\ref{thm:Supercritical} in Section~\ref{subsec::DFcoupling}. For clarity, we first present the proof of Lemma~\ref{lem:Phi2} in Section~\ref{subsec_tree} for a special case where $G$ is an infinite regular tree. This proof is simpler and already contains all the key ideas. 
	Subsequently, in \Cref{subsec_general}, we explain the necessary modifications to extend the proof to general nonamenable transitive unimodular graphs.

	\subsection{The evolving set process and Diaconis-Fill coupling}
	\label{subsec::DFcoupling}
	The whole evolution of the environment is denoted by $\bm{\eta}=(\eta_t: t\ge 0)$.
	We discretize time by observing the random walk at nonnegative integer times. We consider the time-inhomogeneous Markov chain with transition probability given by
	\begin{equation}\label{eq:discre_MC}
		P_{n+1}^{\bm{\eta}}(x, y)=\mathbb{P}^{\bm{\eta}}\left(X_{n+1}=y \;|\; X_{n}=x\right),\quad \forall x, y \in V, \ n \in \mathbb{N} \cup\{0\} .
	\end{equation}
	Note that $\pi(x)\equiv 1, ~{x \in V},$ is a stationary measure for each $P_{n}^{\bm{\eta}}$. Moreover, since the random walk moves at rate 1, we have
	\[P_{n}^{\bm{\eta}}(x, x) \geq e^{-1},\quad \forall x \in V, \ n \in \mathbb{N}.\]
	The evolving set process is a Markov chain that takes values in the collection of subsets of $V$. Its transition is defined as follows: given the current state $S_n=S\subset V$, we pick a random variable $U$ uniformly distributed in $[0,1]$, and the next state of the chain is the set 
	\[S_{n+1}:=\left\{y\in V: \sum_{x\in S}P_{n+1}^{\bm{\eta}}(x, y)\ge U\right\}.\]
	Note that the evolving set process has two absorbing states: $\emptyset$ and $V$.
	Denote by $K_{P}$ the transition probability for the evolving set process $(S_n: n\in \mathbb N\cup \{0\})$ when the transition matrix for the Markov chain is $P\in\{P_n^{\bm{\eta}}: n\in\mathbb{N}\}$. Doob's transform of the evolving set process conditioned to stay nonempty 
	is defined by the transition kernel
	\[\widehat{K}_{P}(A, B)=\frac{\pi(B)}{\pi(A)} K_{P}(A, B).\]
	For more discussion about evolving sets, we refer to \cite{MP05} by Morris and Peres, \cite[Section 6.7]{LP16}, or \cite[Section 17.4]{LP17}.
	
	Now, the Diaconis-Fill coupling $\widehat{\mathbb{P}}^{\bm{\eta}}$ is a coupling between the Markov chain $X_n$ and Doob's transform of the evolving set process, defined as follows. Let $\mathrm{DF}=\{(x, A): x\in A, A\subset V\}$. We define the Diaconis-Fill transition kernel on $\mathrm{DF}$ as 
	\[
	\widehat{P}^{\bm{\eta}}_{n+1}((x, A), (y, B)):=\frac{P_{n+1}^{\bm{\eta}}(x,y) K_{P_{n+1}^{\bm{\eta}}}(A, B)}{\sum_{z\in A}P_{n+1}^{\bm{\eta}}(z,y)},\quad \forall (x, A), (y, B)\in\mathrm{DF}.
	\]
	Let $((X_n, S_n): n\in \mathbb N\cup \{0\})$ be the Markov chain with initial state $(x, \{x\})\in \mathrm{DF}$ and transition kernel $\widehat{P}^{\bm{\eta}}_{n+1}$ from time $n$ to $n+1$. Then, the following properties hold (see Theorem 17.23 of \cite{LP17} for a proof): 
	\begin{itemize}
		\item The chain $(X_n: n\in \mathbb N\cup \{0\})$ has transition kernels $\{P_{n+1}:n\in \mathbb N\cup \{0\}\}$.
		\item The chain $(S_n: n\in \mathbb N\cup \{0\})$ has transition kernels $\{\widehat{K}_{P_{n+1}}: n\in \mathbb N\cup \{0\}\}$.
		\item For any $y\in S_n$, we have 
		\begin{equation} \label{eq:DF_key}
			\widehat{\mathbb{P}}^{\bm{\eta}}(X_n=y\;|\; S_0, S_1, \ldots , S_n)= {|S_n|}^{-1}.
		\end{equation} 
	\end{itemize}
	Throughout the proof, we will write \smash{$\widehat{\mathbb{P}}^{\bm{\eta}}$} for the probability measure arising from the Diaconis-Fill coupling with initial state $(o,\{o\})$ when the entire evolution of the environment is given by $\bm{\eta}$.
	Then, we will use \smash{$\widehat{\mathbb{P}}$} to denote the annealed probability measure with respect to $\bm{\eta}$ when the initial bond configuration is given by $\pi_p$.
	We will use $\widehat{\mathbb{E}}^{\bm{\eta}}$ and $\widehat{\mathbb{E}}$ to denote the corresponding expectations.
	
	For every subgraph $G'$ of $G$ and $S\subset V$, we denote $\partial_{G'} S$ as the \textbf{edge boundary} of $S$ in $G'$, i.e., the set of edges in $E(G')$ that have one endpoint in $S$ and the other endpoint in $V \setminus S$. We will also view $\eta_{t}$ as a subgraph of $G$ with vertex set $V$. We now prove a crucial property related to evolving sets. This property holds for random walks on dynamical percolation of general graphs.

	\begin{lemma}\label{lem:evo}
		Let $(X_t, \eta_t)_{t \geq 0}$ be a random walk on the dynamical percolation of an arbitrary infinite graph $G$, where every vertex has a degree bounded by $d$. Then, the following estimate holds:
		\begin{equation}\label{eq:Sn+1}
			\widehat{\mathbb{E}}^{\bm{\eta}}\left[\left.|S_{n+1}|^{-1/2}\;\right|\;S_n\right]\leq \exp\left(-{\Phi_{S_n}^2}/{6}\right)|S_n|^{-1/2},\quad \forall n\in\mathbb{N}\cup\{0\}.
		\end{equation}
		Here, $\Phi_{S_n}\equiv \Phi^{\bm{\eta}}_{S_{n}}$ is defined as 
		\begin{align}\label{eq.defSn}
			\Phi_{S_{n}} & :=\frac{1}{\left|S_{n}\right|} \sum_{x \in S_{n}} \sum_{y \in S_{n}^{c}} \widehat{\mathbb{P}}^{\bm{\eta}}\left(X_{n+1}=y | X_{n}=x\right),
		\end{align}
		and it satisfies
		\begin{align}\label{eq:Phietan}
			\Phi_{S_n} \geq \frac{1}{ d \ee \left|S_{n}\right|} \int_{n}^{n+1}\left|\partial_{\eta_{t}} S_{n}\right| \d t .
		\end{align}
	\end{lemma}
	\begin{proof} 
		A similar result for random walk on finite graphs has been established in Lemma 2.3 of \cite{PSS20}. Our proof follows a similar approach. By equation (29) of \cite{MP05}, we have
		\begin{align*}
			\widehat{\mathbb{E}}^{\bm{\eta}}\left[\frac{\left|S_{n+1}\right|^{-1 / 2}}{\left|S_{n}\right|^{-1 / 2}} \bigg | S_{n}\right] & =\widehat{\mathbb{E}}^{\bm{\eta}}_{P_{n+1}}\left[\frac{\left|S_{n+1}\right|^{-1 / 2}}{\left|S_{n}\right|^{-1 / 2}} \bigg | S_{n}\right] \\
			& =\mathbb{E}^{\bm{\eta}}_{P_{n+1}}\left[\frac{\left|S_{n+1}\right|}{\left|S_{n}\right|} \frac{\left|S_{n+1}\right|^{-1 / 2}}{\left|S_{n}\right|^{-1 / 2}} \bigg | S_{n}\right] \\
			& =\mathbb{E}^{\bm{\eta}}_{P_{n+1}}\left[\frac{\left|S_{n+1}\right|^{1 / 2}}{\left|S_{n}\right|^{1 / 2}} \bigg | S_{n}\right],
		\end{align*}
		where $\widehat{\mathbb{E}}^{\bm{\eta}}_{P}$ (resp.~$\mathbb{E}^{\bm{\eta}}_{P}$) denotes the expectation when $\{S_n\}$ has transition kernel $\widehat{K}_P$ (resp.~$K_P$), and the first equality is due to the definition of $\widehat{\mathbb{E}}^{\bm{\eta}}$. 
		By Lemma 3 of \cite{MP05} (with $\gamma=\ee^{-1}$) and using the definition \eqref{eq.defSn}, we obtain that
		\[\mathbb{E}^{\bm{\eta}}_{P_{n+1}}\left[\frac{\left|S_{n+1}\right|^{1 / 2}}{\left|S_{n}\right|^{1 / 2}} \bigg | S_{n}\right] \leq 1-\frac{\ee^{-2}}{2\left(1-\ee^{-1}\right)^{2}} \Phi_{S_{n}}^{2} \leq 1-\frac{\Phi_{S_{n}}^{2}}{6}\leq \exp\left[-\frac{\Phi_{S_n}^2}{6}\right] .\]
		This concludes \eqref{eq:Sn+1}.
		
		For the bound \eqref{eq:Phietan},	recall that 
		\[\eta_{t}(x, y)= \begin{cases}1, & \text{ if }\{x, y\} \text { is open in } \eta_{t}, \\ 0, & \text{ if }\{x, y\} \text { is closed in } \eta_{t}.\end{cases}\]
		For neighboring vertices $x$ and $y$, by considering the event that the random walk clock rings exactly once during $\{t\in[n,n+1]:\eta_t(x,y)=1\}$, we have that
		\[\widehat{\mathbb{P}}^{\bm{\eta}}\left(X_{n+1}=y \;|\; X_{n}=x\right)\geq \frac{1}{de}\int_{n}^{n+1} \eta_{t}(x, y) \dd t.\]
		Therefore, we obtain
		\[\sum_{x \in S_{n}} \sum_{y \in S_{n}^{c}} \widehat{\mathbb{P}}^{\bm{\eta}}\left(X_{n+1}=y \;|\; X_{n}=x\right) \geq \frac{1}{de} \sum_{x \in S_{n}} \sum_{y \in S_{n}^{c}} \int_{n}^{n+1} \eta_{t}(x, y) \dd t .\]
		Then, using Fubini's theorem and the fact that
		\[\left|\partial_{\eta_{t}} S_n\right|=\sum_{x \in S_{n}} \sum_{y \in S_{n}^{c}} \eta_{t}(x, y),\]
		we conclude \eqref{eq:Phietan}.
	\end{proof}
	
	We define a sequence of random variables
	\begin{equation}\label{eq:M0n}
		M_0:=1, \quad \text{and}\quad M_n:=|S_n|^{-1/2}\exp\left(\sum_{k=0}^{n-1}\frac{\Phi_{S_k}^2}{6}\right)\quad \forall n\in\mathbb{N}.
	\end{equation}
	Let $\mathscr{F}_{n}$ be the $\sigma$-algebra generated by the evolving sets up to time $n$. By Lemma \ref{lem:evo}, we have
	\[\widehat{\mathbb{E}}^{\bm{\eta}}[M_{n+1}|\mathscr{F}_n]\leq M_n,\]
	which implies that $\{M_n:n\in \mathbb{N}\cup\{0\}\}$ is a supermartingale with respect to $\{\mathscr{F}_n\}$. Thus, we have $\widehat{\mathbb{E}}^{\bm{\eta}}[M_n]\le \widehat{\mathbb{E}}^{\bm{\eta}} [M_0] =1 $. Applying Markov's inequality yields that for any $\varepsilon>0$,
	\begin{equation*}
		\widehat{\mathbb{P}}^{\bm{\eta}}(M_n \geq 2/\varepsilon)\leq \frac{\varepsilon }{2}\mathbb{E}^{\bm{\eta}}[M_n]=\frac{\varepsilon}{2}.
	\end{equation*}
	In the special case where the initial configuration is $\pi_p$, it also gives that
	\begin{equation}\label{eq:Mnlb}
		\widehat{\mathbb{P}}(M_n \geq 2/\varepsilon)\leq \varepsilon/2.
	\end{equation}
	
	\begin{lemma}\label{lem:Phi2}	
		Under the setting of Theorem~\ref{thm:Supercritical}, suppose $\eta_{0}$ has the stationary distribution $\pi_{p}$. Then, there exist constanats $c_0, c_1>0$ such that for any $n \in \mathbb{N} \cup\{0\}$, 
		\begin{equation}\label{eq:PSn_G}		
			\widehat{\mathbb{P}}\left( \int_{n}^{n+1}\left|\partial_{\eta_{t}} S_n\right| \dd t  \geq c_1 |S_n| \right) \geq c_0. 
		\end{equation}									
	\end{lemma}

	This lemma will be proved as \Cref{lem:Phi} in the tree case, and will be proved at the end of \Cref{subsec_general} in the general case. 
	Assuming Lemma~\ref{lem:Phi2}, we first show that the volume of $S_n$ grows exponentially fast in Lemma~\ref{lem::ES_exp}, which is then used to complete the proof of Theorem~\ref{thm:Supercritical}. 
	
	\begin{lemma}\label{lem::ES_exp}
		Using the same notations as in Lemma~\ref{lem:Phi2}, we have  
		\begin{equation}\label{eqn::ESP_keyestimate_G}
			\widehat{\mathbb{P}}\left(|S_n|>\left(\frac{c_0}{4}\right)^{2} \exp\left[\frac{c_0c_1^2 }{6 e^2d^2}n\right]\right)\geq \frac{c_0}{4},\quad \forall n\in\mathbb{N}.
		\end{equation}
	\end{lemma}
	
	\begin{proof}
		We define the indicator functions
		\[I_{k}:= \mathbf 1(\Phi_{S_{k}} \geq c_1/(ed))\]
		Then, by \eqref{eq:Phietan} and \eqref{eq:PSn_G}, we have that
		\begin{align}\label{eq:EIk}
			\widehat{\mathbb{E}}[I_k]&=\widehat{\mathbb{P}}\left(\Phi_{S_{k}} \geq c_1/(ed)\right)\geq\widehat{\mathbb{P}}\left( \int_k^{k+1}|\partial_{\eta_t} S_k| \dd t \geq c_1|S_k| \right)\geq c_0,\quad  \forall k\in\mathbb{N}\cup\{0\}.
		\end{align}
		Therefore, we have that 
		\begin{align*}
			c_0 n\leq \widehat{\mathbb{E}}\left[\sum_{k=0}^{n-1}I_k\right]&=\widehat{\mathbb{E}}\left[\sum_{k=0}^{n-1}I_k\cdot \mathbf{1}_{\left\{\sum_{k=0}^{n-1}I_k\geq \frac{c_0 n}{2}\right\}}\right]+\widehat{\mathbb{E}}\left[\sum_{k=0}^{n-1}I_k\cdot \mathbf{1}_{\left\{\sum_{k=0}^{n-1}I_k<\frac{c_0 n}{2}\right\}}\right]\\
			&\leq n\widehat{\mathbb{P}}\left(\sum_{k=0}^{n-1}I_k\geq \frac{c_0 n}{2}\right)+\frac{c_0 n}{2},
		\end{align*}
		which implies that
		\begin{equation}\label{eq:sumI_k}
			\widehat{\mathbb{P}}\left(\sum_{k=0}^{n-1}I_k\geq \frac{c_0 n}{2}\right)\geq \frac{c_0}{2},\quad \forall n\in\mathbb{N}.
		\end{equation}
		Taking $\varepsilon={c_0}/{2}$ in \eqref{eq:Mnlb}, we get that
		\[\widehat{\mathbb{P}}\left(|S_n|>\left(\frac{c_0}{4}\right)^{2}\exp\left[\sum_{k=0}^{n-1}\frac{\Phi_{S_k}^2}{3}\right]\right)\geq 1-\frac{c_0}{4},\quad \forall n\in\mathbb{N}.\]
		Combined with with \eqref{eq:sumI_k}, this implies that
		\[\widehat{\mathbb{P}}\left(|S_n|>\left(\frac{c_0}{4}\right)^{2} \exp\left[\sum_{k=0}^{n-1}\frac{\Phi_{S_k}^2}{3}\right]\text{ and }\sum_{k=0}^{n-1}I_k\geq \frac{c_0 n}{2}\right)\geq \frac{c_0}{4},\]
		which concludes \eqref{eqn::ESP_keyestimate_G} by the definition of $I_k$.
	\end{proof}
	
	\begin{proof}[Proof of Theorem~\ref{thm:Supercritical}]
		Recall that $B(o,k)$ denotes the $k$-neighborhood of $o$ on $G$.
		Let $((X_n, S_n): n\ge 0)$ be the Diaconis-Fill coupling defined above. By \eqref{eq:DF_key}, we have 
		\begin{equation*}	\widehat{\mathbb{P}}^{\bm{\eta}}\left(X_n\in B(o,k)\;\Big |\; |S_n|> (8/c_0)d^{k} \right)\leq\frac{|B(o,k)|}{(8/c_0)d^{k}}\leq\frac{c_0}{8},
		\end{equation*}
		which implies that 
		\begin{align*}
			\widehat{\mathbb{P}}^{\bm{\eta}}\left(X_n\in B(o,k)\right)&\leq \widehat{\mathbb{P}}^{\bm{\eta}}\left(X_n\in B(o,k)\;\Big |\; |S_n|> (8/c_0)d^{k} \right)+\widehat{\mathbb{P}}^{\bm{\eta}}\left(|S_n|\leq (8/c_0)d^{k}\right)\nonumber\\
			&\leq\widehat{\mathbb{P}}^{\bm{\eta}}\left(|S_n|\leq (8/c_0)d^{k}\right)+\frac{c_0}{8}.
		\end{align*}
		In the special case where the initial configuration is $\pi_p$, we obtain that
		\begin{equation}\label{eqn::EPS_aux}
			\widehat{\mathbb{P}}\left(X_n\in B(o,k)\right)\leq \widehat{\mathbb{P}}\left(|S_n|\leq (8/c_0)d^{k}\right)+\frac{c_0}{8}.
		\end{equation}
		By setting
		\[k=k(n)=\left\lfloor\frac{1}{\log d}\left(\log\frac{c_0^3}{128}+\frac{c_0c_1^2}{6e^2d^2}n\right)\right\rfloor\]
		in~\eqref{eqn::ESP_keyestimate_G}, 
		we get that
		\[\widehat{\mathbb{P}}\left(|S_n|\leq (8/c_0)d^{k(n)}\right)\leq 1-\frac{c_0}{4} \quad \forall n\in\mathbb{N}.\]
		Combined  with~\eqref{eqn::EPS_aux}, this gives us 
		\begin{align}\label{eq:XinBall}
			\widehat{\mathbb{P}}\left(X_n\in B(o,k(n))\right)\leq  1-\frac{c_0}{8}.
		\end{align}
		Applying Lemma~\ref{lem:existence_speed} and \eqref{eq:XinBall} in the case where $\eta_0$ has distribution $\pi_p$, we otain
		\begin{equation}\label{eqn::super_lowerbound_G}
			v_p(\mu)\ge \frac{c_0c_1^2}{6e^2d^2\log d}. 
		\end{equation}
		Finally, by Proposition~\ref{prop:stationary_arbitrary}, this bound extends to a general initial bond configuration.	 
	\end{proof}
	
	\subsection{The tree case}\label{subsec_tree}

	Let $G=\mathbb{T}_{d}=(V\left(\mathbb{T}_{d}\right),E\left(\mathbb{T}_{d}\right))$ be an infinite regular tree where every vertex has degree $d\ge 3$. It is known that $p_c(\mathbb{T}_{d})=b^{-1}$, where $b:=d-1$. 
	Let $o$ be an arbitrary vertex in $V$. Define
	\begin{equation*}
		\theta_p:=\mathbb{P}(\exists\text{ infinite open path containing } o  \text{ in the Bernoulli-$p$ bond percolation of } \mathbb{T}_{d}).
	\end{equation*}
	Let $\tilde{\mathbb{T}}_{b}$ be an infinite $b$-ary tree, i.e., an infinite tree with a root $o$ having degree $b$ and all other vertices having degree $d$. Then, define
	\begin{equation*}
		\tilde{\theta}_p:=\mathbb{P}(\exists\text{ infinite open path containing } o  \text{ in the Bernoulli-$p$ bond percolation of } \tilde{\mathbb{T}}_{b}).
	\end{equation*}
	Note that $\theta_p$ and $\tilde{\theta}_p$ are related by the equation
	\begin{equation*}
		1-\theta_p=(1-p\tilde{\theta}_p)^{d}. 
	\end{equation*}
	In particular, we have 
	$\theta_p\sim dp\tilde{\theta}_p\text{ as }p\downarrow p_c$.
	This subsection focuses on proving the following theorem, which implies the desired lower bound in \eqref{eq:Supercritical}. 
	
	
	
	\begin{theorem}\label{thm::supercritical_quantitative}
		For any $p>p_c=1/b$, the speed for the random walk on dynamical percolation in $\mathbb{T}_{d}$ with any initial bond configuration $\eta_0$ satisfies
		\begin{equation}\label{eq:lowerbound_tree}
			v_p(\mu)\geq\frac{p^3(p\tilde{\theta}_p)^9}{48 e^2d^3\log d}.
		\end{equation}
		In particular, there exists a constant $c_{\ref{thm::supercritical_quantitative}}=c_{\ref{thm::supercritical_quantitative}}(d)>0$ depending only on $d$ such that	\begin{equation}\label{eq:lowerv9}
			v_p(\mu)\geq c_{\ref{thm::supercritical_quantitative}}(p-p_c)^9 \ \ \text{ as }\ \ p\downarrow p_c.\end{equation}
	\end{theorem}
	Recall that in the context of bond percolation $\eta$, a vertex $x\in V$  is called a \textbf{trifurcation point} of $\eta$ if closing all edges incident to $x$ would split the component of $x$ in $\eta$ into at least 3 disjoint infinite connected components; see Figure \ref{fig:tri} for an illustration of trifurcation points on $\mathbb T_3$. 
	
	\begin{figure}
		\begin{center}
			\includegraphics{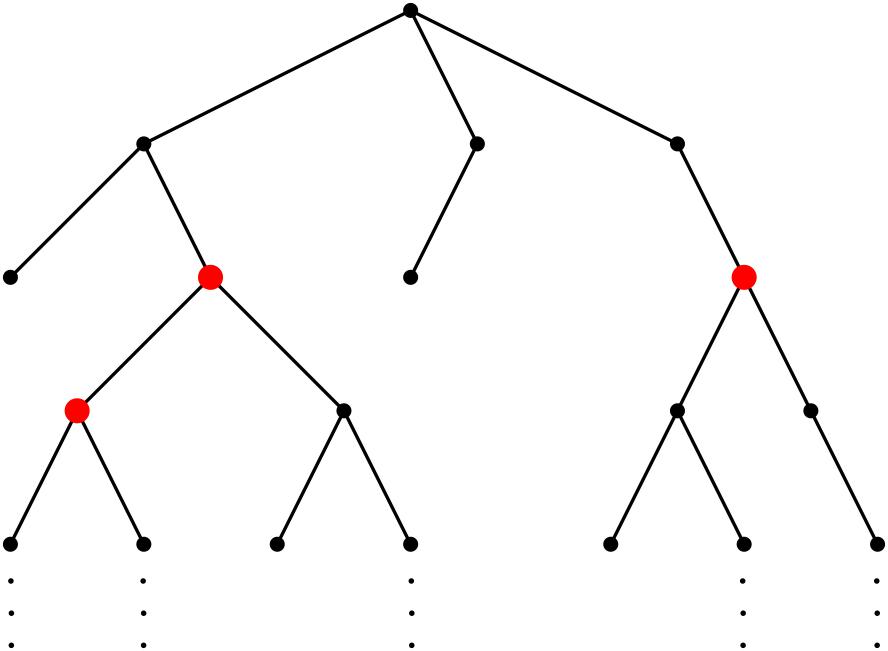}
			\caption{Dots stand for an infinite open path emanating from the corresponding vertex; three trifurcation points are marked in red.}\label{fig:tri}
		\end{center}
	\end{figure}

	Assuming that $\eta_{0}$ is distributed according to $\pi_{p}$, by stationarity, $\eta_t$ also follows the distribution $\pi_{p}$ for any $t\ge 0$. Let $S \subset V\left(\mathbb{T}_{d}\right)$ be an arbitrary finite subset. For each $x \in S$, by setting $3$ edges incident to $x$ to be open, we have
	\begin{equation}\label{eq:trif_estimate}
		\widehat{\mathbb{P}}(x \text{ is a trifurcation point of } \eta_t)\geq (p\tilde{\theta}_p)^3.  
	\end{equation}
	This implies that
	\begin{equation}\label{eq:numberoftri}
		\widehat{\mathbb{E}}[\text{number of trifurcation points of }  \eta_t \text{ in } S]\geq |S|(p\tilde{\theta}_p)^3.
	\end{equation}
	For any $S \subset V$, Burton-Keane~\cite{BK1989} proved (which can be verified by induction on $|S|$) that
	\begin{equation}\label{eq:triine}
		|\partial_{\eta_t} S|\geq (\text{number of trifurcation points of $\eta_t$ in }S)+2.
	\end{equation}
	Combining \eqref{eq:numberoftri} and \eqref{eq:triine}, we get that
	\begin{equation}\label{eq:ebSb}
		\widehat{\mathbb{E}}\left[|\partial_{\eta_t} S|\right]\geq  |S|(p\tilde{\theta}_p)^3,\quad  \forall t\geq 0.
	\end{equation}
	With this estimate, we can derive the following lemma. 
	\begin{lemma}\label{lem:Phi}
		Suppose $\eta_{0}$ follows the stationary distribution $\pi_{p}$. For any fixed finite nonempty subset $S \subset V\left(\mathbb{T}_{d}\right)$ and $n \in \mathbb{N} \cup\{0\}$, we have
		\begin{equation}\label{eq:PSn}
			\widehat{\mathbb{P}}\left( \frac{1}{|S|}\int_{n}^{n+1}\left|\partial_{\eta_{t}} S\right| \dd t  \geq \frac{1}{2}(p \tilde{\theta}_p)^{3}\right) \geq \frac{(p \tilde{\theta}_p)^{3}}{2d}.
		\end{equation}
		A similar estimate also holds when $S=S_n$ (which is random and depends on the environment $\eta_n$):
		\begin{equation}\label{eq:PSnSn}
			\widehat{\mathbb{P}}\left( \frac{1}{|S_n|}\int_{n}^{n+1}\left|\partial_{\eta_{t}} S_n\right| \dd t  \geq \frac{p}{2}(p \tilde{\theta}_p)^{3}\right) \geq \frac{p}{2d}(p \tilde{\theta}_p)^{3}.
		\end{equation}
	\end{lemma}
	\begin{proof}
		First, we trivially have $\left|\partial_{\eta_{t}} S\right| \leq d|S|$, which implies the rough bound
		\[Z_n :=\frac{1}{|S|}\int_{n}^{n+1}\left|\partial_{\eta_{t}} S\right| \d t \le d.\]
		On the other hand, by Fubini's theorem and \eqref{eq:ebSb}, we have
		\begin{equation*}
			\widehat{\mathbb{E}}[Z_n]=\frac{1}{|S|}\int_{n}^{n+1} \widehat{\mathbb{E}}\left[\left|\partial_{\eta_{t}} S\right|\right] \d t \geq (p \tilde{\theta}_p)^{3}.
		\end{equation*}
		Thus, we get that
		\begin{align}
			(p \tilde{\theta}_p)^{3} \leq & \widehat{\mathbb{E}}[Z_n] 
			= \widehat{\mathbb{E}}\left[Z_n \mathbf{1}\left(Z_n \geq \frac{1}{2}{(p \tilde{\theta}_p)^{3}}\right)\right] +\widehat{\mathbb{E}}\left[Z_n \mathbf{1}\left(Z_n<\frac12{(p \tilde{\theta}_p)^{3}} \right)\right] \nonumber\\
			\leq & d  \widehat{\mathbb{P}}\left(Z_n \geq \frac12{(p \tilde{\theta}_p)^{3}}\right)+\frac12 (p \tilde{\theta}_p)^{3},\label{eq:trivialYn}
		\end{align}
		which gives the desired inequality \eqref{eq:PSn}.
		
		To show the estimate for $S_n$, we will establish the lower bound
		\begin{equation}
			\label{eq:ESn} \widehat{\mathbb{E}}\left[  \frac{|\partial_{\eta_t} S_n|}{|S_n|} \right]\geq  p(p\tilde{\theta}_p)^3 ,\quad  \forall t\in [n,n+1].
		\end{equation}
		Then, using a similar argument as above, we can conclude \eqref{eq:PSnSn}. Since for $t\in [n,n+1]$, 
		$$\widehat{\mathbb{E}}\left[ \left. |\partial_{\eta_t} S_n|\right| S_n \right]\ge [e^{-\mu(t-n)} + p(1-e^{-\mu(t-n)})]\widehat{\mathbb{E}}\left[ \left. |\partial_{\eta_n} S_n|\right| S_n \right]\ge p \widehat{\mathbb{E}}\left[ \left. |\partial_{\eta_n} S_n|\right| S_n \right],$$ 
		it suffices to prove that 
		\begin{equation}
			\label{eq:ESnSn} \widehat{\mathbb{E}}\left[ \frac{ |\partial_{\eta_n} S_n|}{|S_n|} \right]\geq (p\tilde{\theta}_p)^3 .
		\end{equation}
		By \eqref{eq:triine}, it follows from the following estimate on the proportion of trifurcation points:
		\begin{equation}\label{eq:numberoftri2}	\widehat{\mathbb{E}}\left[\frac{1}{|S_n|}\sum_{x\in S_n} \mathbf 1(x \text{ is a trifurcation point of }\eta_n) \right]\geq (p\tilde{\theta}_p)^3.
		\end{equation}
		Using \eqref{eq:DF_key}, we obtain that 
		\begin{align*}
			&~\frac{1}{|S_n|}\widehat{\mathbb E} ^{\bm{\eta}}\left[ \sum_{x\in S_n} \mathbf 1{(x \text{ is a trifurcation point of }\eta_n)}\Big| S_n\right] \\
			= &~ \widehat{\mathbb P} ^{\bm{\eta}}\left( \left. X_n \text{ is a trifurcation point of }\eta_n\right| S_n\right). 
		\end{align*}
		Taking the expectation of both sides, we see that to prove \eqref{eq:numberoftri2}, it suffices to show that 
		\begin{equation}\label{eq:numberoftri3}
			\widehat{\mathbb P}\left(X_n \text{ is a trifurcation point of }\eta_n\right) \geq (p\tilde{\theta}_p)^3.
		\end{equation}
		By \Cref{lem.ViewStationary}, the environment seen by the moving particle is stationary when $\eta_0$ has distribution $\pi_p$. Thus, we have that 
		\begin{equation}\label{eq:numberoftri4}
			\widehat{\mathbb P}\left(X_n \text{ is a trifurcation point of }\eta_n \right)= \widehat{\mathbb P}\left( X_0 \text{ is a trifurcation point of }\eta_0\right) \geq (p\tilde{\theta}_p)^3.
		\end{equation}
		This implies \eqref{eq:numberoftri3} and concludes the proof of \eqref{eq:PSnSn}.
	\end{proof}

	\begin{proof}[Proof of Theorem~\ref{thm::supercritical_quantitative}]
		By Lemma~\ref{lem:Phi}, we can take 
		\[c_0=\frac{p(p \tilde{\theta}_p)^{3}}{2d}, \quad c_1=\frac{p(p \tilde{\theta}_p)^{3}}{2},\]
		in \eqref{eq:PSn_G}. Then, applying \eqref{eqn::super_lowerbound_G} concludes  \eqref{eq:lowerbound_tree}. 
		It is well-known that (see e.g., equation (2.1.17) of \cite{HeydenreichvanderHofstadsurvey} by Heydenreich and van der Hofstad)
		\begin{equation}\label{eq:wtThetap}
			\tilde{\theta}_p\sim \frac{2b^2}{b-1}(p-p_c) \ \ \text{ as }\ \ p\downarrow p_c.
		\end{equation}
		Combined with~\eqref{eq:lowerbound_tree}, it implies~\eqref{eq:lowerv9}. 
	\end{proof}
	\subsection{The general case}\label{subsec_general}
	
	In this subsection, we complete the proof of Lemma~\ref{lem:Phi2} for general nonamenable transitive unimodular graphs. 
	We still use the evolving set process and the Diaconis-Fill coupling defined for the general graph $G$. 
	The proof of \Cref{lem:Phi2} depends on the following lemma established by Benjamini, Lyons, and Schramm \cite{BLS1997}. Hermon and Hutchcroft \cite{HermonHutchcroft2021} also utilized this lemma (stated as Lemma 2.5 there) to prove the anchored expansion of percolation clusters. 
	
	\begin{lemma}\label{lem:trifucation_general}
		Let $G$ be a connected, locally finite, nonamenable transitive unimodular graph. We select an arbitrary root $o$ in $G$. For any $p \in (p_c,1]$, there exists an automorphism-invariant percolation process $\zeta$ on $G$ and a coupling between $\zeta$ and the Bernoulli-$p$ bond percolation $\eta$ (with law $\pi_p$) on $G$, such that the following holds:
		\begin{itemize}
			\item[(i)] There exists a constant $c_p\equiv c_p(G)>0$ such that the root $o$ is a trifurcation point of $\zeta$ with probability at least $c_p$.
			\item[(ii)] The process $\zeta$ is dominated by $\eta$, i.e., $\zeta\le \eta$.
			\item[(iii)] The coupling $(\eta,\zeta)$ between $\eta$ and $\zeta$ is automorphism-invariant.
			Specifically, $\zeta$ can be obtained as $\zeta=F(\eta,\xi)$, where $\xi$ is a collection of i.i.d.~random variables at the vertices and edges (independent of $\eta$), and $F$ is an equivariant function under automorphisms of $G$. 
		\end{itemize}
	\end{lemma}
	\begin{proof}
		Under the assumption of unimodularity, this lemma is essentially contained in Lemma 3.8 and Theorem 3.10 of \cite{BLS1997}. 
		(It also holds for nonunimodular graphs as explained in Lemma 2.5 of \cite{HermonHutchcroft2021}.) We now briefly explain how the equivariant function $F$ is constructed.
		Let $$p_u=p_u(G):=\inf\{p\in[0,1]:\eta \text{ has a unique infinite cluster } \P_p\text{-a.s.}\}\,.$$
		If $p_c<p_u$, we can choose $p_*\in (p_c, p_u\wedge p)$ and let $\xi$ be a Bernoulli-$(p_*/p)$ percolation on $G$, independent of $\eta$. Then, we can define the function $F$ as   $$\zeta(e)=F(\xi,\eta)(e):=\xi(e)\eta(e).$$
		In the more challenging case $p_c=p_u$, the construction of $F$ was explained in the proof of Lemma 3.8 and Theorem 3.10 in \cite{BLS1997}. It involves Bernoulli percolation (independent of $\eta$), minimal spanning forest, and wired uniform spanning forest. (To construct the wired uniform spanning forest, one can employ the classical Wilson's method, initially developed by Wilson \cite{Wilson_Tree} and subsequently extended by Benjamini, Lyons, Peres, and Schramm \cite{BLPS:2001_AOP} to infinite graphs.) 
		The proof in \cite[Theorem 3.10]{BLS1997} is based on the construction of a subprocess $\eta'\subset \eta$ in \cite[Theorem 3.1]{BLS1997}. 
		There is one minor point in the construction of $\eta'$ that we want to clarify: when a shortest path between two points must be selected in an equivariant way, we assign to the edges an independent collection of i.i.d.~uniform $[0,1]$ labels and choose the path with the minimal sum. This is slightly different from the construction in the proof of \cite[Theorem 3.1]{BLS1997}.
	\end{proof}

	\begin{proof}[Proof of \Cref{lem:Phi2}]
		We prove that \eqref{eq:PSn_G} holds for 
		$$ c_0=\frac{pc_p}{2d},\quad c_1=\frac{pc_p}{2},$$
		where $c_p$ is the constant in \Cref{lem:trifucation_general} (i). 
		Similar to the proof of \eqref{eq:PSnSn}, we only need to establish the following estimate on $|\partial_{\eta_n} S_n|/|S_n|$, which is analogous to \eqref{eq:ESnSn}:
		\begin{equation} \nonumber
			\widehat{\mathbb{E}}\left[  {|\partial_{\eta_n} S_n|}/{|S_n|} \right]\geq c_p. 
		\end{equation}
		Using the function $F$ from \Cref{lem:trifucation_general}, we define $\zeta_t=F(\eta_t,\xi)$ for $t\ge 0$. By \Cref{lem:trifucation_general} (ii), it suffices to prove that 
		\begin{equation}\label{eq:ESnSn_G}
			\widehat{\mathbb{E}}\left[  {|\partial_{\zeta_n} S_n|}/{|S_n|} \right]\geq c_p. 
		\end{equation}
		Following the argument presented below \eqref{eq:ESnSn}, we can derive that 
		\begin{align*}
			\widehat{\mathbb{E}}\left[  \frac{|\partial_{\zeta_n} S_n|}{|S_n|} \right] \ge \widehat{\mathbb P}\left(X_n \text{ is a trifurcation point of }\zeta_n\right).
		\end{align*}
		By \Cref{lem.ViewStationary}, the pair $(\eta_t,\zeta_t)=(\eta_t, F(\eta_t,\xi))$ seen by the moving particle at $X_t$ is stationary. Hence, we have  
		$$\widehat{\mathbb P}\left(X_n \text{ is a trifurcation point of }\zeta_n\right)=\widehat{\mathbb P}\left(X_0 \text{ is a trifurcation point of }\zeta_0\right)\ge c_p$$
		according to \Cref{lem:trifucation_general} (i). This concludes \eqref{eq:ESnSn_G}, and hence completes the proof of \Cref{lem:Phi2}. 
	\end{proof}

	\section{Speed for the subcritical case}\label{sec_sub}
	
	
	In this section, we provide the proofs of  \Cref{thm:Subcritical} and \Cref{prop:general_Lower}. Note that when $p<p_c$, \Cref{prop:general_Lower} already gives the desired lower bound in \eqref{eq:Subcritical}, so we focus on proving the upper bound in \Cref{thm:Subcritical} and \Cref{prop:general_Lower}.
	
	\subsection{Proof of \Cref{thm:Subcritical}}
	

	The upper bound in \eqref{eq:Subcritical} is trivially satisfied when $\mu \ge 1/2$. Hence, in the following proof, we will assume that $\mu\le 1/2$. It suffices to establish the following result.	
	
	\begin{theorem}\label{thm:Subcritical_Upper}
		For any $p\in (0,p_c)$ and $\mu \in (0,1/2]$, there exists a constant $C_{\ref{thm:Subcritical_Upper}}>0$ independent of $\mu$ such that the following estimate holds for all $t\ge 0$:
		\begin{align}
			\e[\mathrm{dist}(X_0, X_{t/\mu})] \leq C_{\ref{thm:Subcritical_Upper}}   (t \vee 1).
		\end{align}
	\end{theorem}

	This theorem is a consequence of the following classical exponential tail estimate \eqref{eq:expdecay} concerning the diameter of connected components in subcritical percolation. This estimate follows directly from \cite[Theorem 1.1]{Duminil-Copin:2016CMP} by Duminil-Copin and Tassion for subcritical percolation on arbitrary locally finite transitive infinite graphs. In fact, it already follows from the arguments in the breakthrough works of Menshikov \cite{Menshikov1986} and Aizenman and Barsky \cite{AizenmanBarsky:1987CMP}.

	\begin{proposition}\label{thm:exp.volume}
		For any $p\in (0,p_\critical)$, there exists a constant $C_{\ref{thm:exp.volume}}=C_{\ref{thm:exp.volume}}(p)>0$ such that for any vertex $o\in V$ and all $r>0$, the following bound holds:
		\begin{equation}\label{eq:expdecay}
			\P_p( \Rade(\mathcal C_o)\ge r)\le e^{-C_{\ref{thm:exp.volume}}r}.			    
		\end{equation}
		Here, recall that $\mathcal C_o$ is the connected component containing $o$, and $\Rade(\cdot)$ refers to its extrinsic radius as defined in \eqref{eq.defRadius}.   
	\end{proposition}

	\begin{proof}[Proof of Theorem~\ref{thm:Subcritical_Upper}]
		We divide the time interval $[0, t/\mu]$ into smaller time intervals of length $\beta / \mu$, where $\beta>0$ is a constant that will be chosen later. Applying the triangle inequality, we get that 	\begin{align}\label{eq.subUpper}
			\dist(X_0, X_{t/\mu}) \leq \sum_{k=0}^{\lfloor t/\beta \rfloor - 1 } \dist(X_{k \beta/ \mu}, X_{(k+1)\beta/\mu}) + \dist(X_{\lfloor t/\beta \rfloor \cdot \beta/\mu}, X_{t/\mu}). 
		\end{align} 
		For $0\le k \le \lfloor t/\beta \rfloor$, let $\bar \eta_k$ be the set of edges that are open at some point  during $[k \beta/ \mu,  (k+1)\beta/\mu]$. Let $\clt_{k} \subset V$ denote the open cluster of $X_{k \beta/ \mu}$ with respect to the bond configuration $\bar \eta_k$.   
		Then, the particle must stay within $\clt_{k}$ during $[k \beta/ \mu, (k+1)\beta/\mu]$, giving that  	\begin{align}\label{eq.subRange}
			\dist(X_{k \beta/ \mu}, X_{(k+1)\beta/\mu}) \leq 2\Rade(\clt_k).
		\end{align}
		By combining \eqref{eq.subUpper} and \eqref{eq.subRange}, we conclude that
		\begin{align}\label{eq.subExpt}
			\e[\dist(X_0, X_{t/\mu})] \leq 2\sum_{k=0}^{\lfloor t/\beta \rfloor} \e[\Rade(\clt_k)].
		\end{align}

		Note that for 
		any edge $e$, we have 
		\begin{align*}
			\tilde p:=\p\left(e \text{ is open some time during } [k \beta/ \mu, (k+1)\beta/\mu]\right) \le 1- (1-p) e^{- \beta}.
		\end{align*}
		Since $p < p_c$, we can choose a small enough constant $\beta>0$ such that $\tilde{p}<p_c$.
		Hence, $\clt_k$ is a connected component in the subcritical percolation of $G$ with parameter $\tilde p$.  
		By \Cref{thm:exp.volume}, $\Rade(\clt_k)$ has exponential tail, which implies that $\e[\Rade(\clt_k)]\le C$ for a large constant $C>0$. Together with \eqref{eq.subExpt}, it concludes the proof.  
	\end{proof}
	
	\subsection{A general lower bound: Proof of \Cref{prop:general_Lower}}

	In this subsection, we prove \Cref{prop:general_Lower}. As a special case, when $p<p_c$, it gives the lower bound in \eqref{eq:Subcritical}. 
	
	\subsubsection{Case $\mu > 1/2$.}
	First, we consider the simpler case where $\mu > 1/2$. Similarly to the setting in \Cref{subsec::DFcoupling}, we utilize the Diaconis-Fill coupling between the random walk and the evolving set process. Following the proof of Theorem~\ref{thm:Supercritical}, our goal is to establish a similar estimate as in Lemma~\ref{lem:Phi2}. 
	Let $\p_{\eta_0}$ denote the probability measure of the environment process $\bm{\eta}$ when the initial bond configuration is ${\eta}_0$. 
	\begin{lemma}\label{lem:largemu}
		Assume the same setup as in \Cref{prop:general_Lower}. For any $p\in (0,1]$, $\mu > 1/2$, initial environment configuration ${\eta}_0$, and nonempty finite subset $S\subset V$, we have that 
		\begin{equation}\label{eq:partialSnn}
			\p_{{\eta}_0}\left(	\frac{1}{|S|}\int_{0}^{1}\left|\partial_{\eta_{t}} S\right| \dd 	t								 \geq \frac{cpd}{2}\Phi(G)\right)\geq \frac{cp}{2},
		\end{equation}
		where $c	=\int_0^1(1-e^{-t/2}) \dd t	>0$.			
	\end{lemma}
	\begin{proof}
		Recall that 
		\[\left|\partial_{\eta_t}S\right|=\sum_{x\in S}\sum_{y\in S^c}\eta_t(x,y)\le \left|\partial_E S\right|.\]
		For each edge $e\in \partial_E S$, we have
		\begin{equation}\label{Peta0}
			\p_{{\eta}_0}(\eta_t(e)=1)\ge (1-e^{-\mu t})p ,
		\end{equation}
		where the equality holds when $\eta_0(e)=0$. Let $\e_{{\eta}_0}$ denote the expectation with respect to $\p_{{\eta}_0}$. By applying Fubini's theorem and \eqref{Peta0}, we obtain that
		\[\e_{{\eta}_0}\left[ \int_{0}^{1}\left|\partial_{\eta_{t}} S\right| \dd t \right]=  \int_{0}^{1}\e_{{\eta}_0} \left[\left|\partial_{\eta_{t}} S\right|\right] \dd t \ge |\partial_E S|\cdot \int_0^1 (1-e^{-\mu t})p \dd t\ge 			c p |\partial_E S|.\]	
		Then, from this inequality, we get that 
		\begin{align*}
			cp |\partial_E S| \leq & \e_{{\eta}_0}\left[ \int_{0}^{1}\left|\partial_{\eta_{t}} S\right| \dd t\right] \le |\partial_E S| \p_{\eta_0}\left(\int_{0}^{1}\left|\partial_{\eta_{t}} S\right| \dd t \geq \frac{	 cp}{2}|\partial_E S|\right)+\frac{		 c p }{2}|\partial_E S|,
		\end{align*}	
		which implies that 
		\begin{align*}
			\p_{\eta_0}\left(\frac{1}{|S	|}\int_{0}^{1}\left|\partial_{\eta_{t}} S\right| \dd t \geq \frac{	 cp}{2}\frac{|\partial_E S|}{|S|}\right)\ge \frac{					 cp}{2}	.					
		\end{align*}
		This estimate, together with the fact that $|\partial_E S|\ge d\Phi(G)|S|$, gives the desired inequality \eqref{eq:partialSnn}. 
	\end{proof}
	
	\begin{corollary}\label{cor:smallmu}
		Under the setting of \Cref{lem:largemu}, we have that  
		\begin{equation}		\label{eq:lowerbound_general_largemu}		
			v_p(\mu)\ge 	\frac{c^3p^3\Phi(G)^2}{48e^2\log d}.										
		\end{equation}
	\end{corollary}
	\begin{proof}
		With \eqref{eq:partialSnn}, we obtain~\eqref{eq:PSn_G} with the following values: 
		\[c_0=\frac{cp}{2}, \quad c_1=\frac{cpd}{2}\Phi(G).\]
		Using exactly the same argument as in Section~\ref{subsec::DFcoupling} and setting $c_0, c_1$ in~		\eqref{eqn::super_lowerbound_G} as above, we can conclude~\eqref{eq:lowerbound_general_largemu}. 								 \end{proof}

	\subsubsection{Case $\mu \in (0,1/2]$.}
	
	It remains to deal with the more challenging case where $\mu \in (0,1/2]$. In this case, we discretize time by observing the random walk at times $n/\mu$ for $n\in \N$. Similarly to \eqref{eq:discre_MC}, we consider another time-inhomogeneous Markov chain $Y_n:=X_{n/\mu}$ for $n \in \mathbb{N} \cup\{0\}$. With a slight abuse of notation, we continue to denote the transition probability by
	\begin{equation}\label{eq:discre_MC2}
		P_{n+1}^{\bm{\eta}}(x, y)=\mathbb{P}^{\bm{\eta}}\left(Y_{n+1}=y \;|\; Y_{n}=x\right),\quad \forall x, y \in V,\ n \in \mathbb{N} \cup\{0\} .
	\end{equation}
	Then, we define the Diaconis-Fill coupling and the evolving set process using $P_{n+1}^{\bm{\eta}}$ as in \Cref{subsec::DFcoupling}. Moreover, we adopt exactly the same notations as those below \eqref{eq:discre_MC}, replacing the $X_n$'s with $Y_n$. Similarly to the proof of Theorem~\ref{thm:Supercritical}, it suffices to establish the following counterpart of the estimate \eqref{eq:EIk}: there exist constants $c_0, c_1>0$ such that  
	\begin{equation}\label{eq:PSn_G_lowerbound_musmall}		
		\widehat{\mathbb{P}}\left( \Phi_{S_n}\ge c_1/(ed) \right) \geq c_0,\quad \forall n \in \mathbb{N} \cup\{0\}. 
	\end{equation}
	To this end, we need the following two auxiliary lemmas whose proofs are similar to those of Lemmas 3.1 and 3.2 in \cite{PSS18}.


	\begin{lemma}\label{lem:evo4}
		The following estimate holds for all nonempty finite subsets $S\subset V$, $p\in (0,1]$, $\mu\in (0,1/2]$, and any initial environment configuration ${\eta}_0$: 
		\begin{equation}\label{eq:size_partialS}
			\p_{\eta_0}\left(   \left|\{e\in \partial_E S: \eta_t(e) =1 \ \text{for all} \ t\in [\mu^{-1}-1,\mu^{-1}]\}\right|\ge \beta |\partial_E S| \right) \ge \frac{\beta}{2},
		\end{equation} 
		where $\beta=\frac{p}{2}(1-e^{-1/2})e^{-(1-p)/2}$. 
	\end{lemma} 
	\begin{proof}
		Note that the LHS is minimized when the initial configuration is $\eta_0(e)\equiv 0$ for $e\in E$. Thus, it is bounded from below by the probability that $|\text{Bin}(|\partial_E S|, (1-e^{-1+\mu})pe^{-\mu(1-p)})|\ge \beta |\partial_E S| $. Here, $\text{Bin}(n, q)$ denotes a binomial random variable with parameters $n$ and $q$. Applying the Paley–Zygmund inequality and noticing that $(1-e^{-1+\mu})pe^{-\mu(1-p)}\ge 2\beta$ for all $\mu\in (0,1/2]$, we get \eqref{eq:size_partialS}. 
	\end{proof}
	
	\begin{lemma}\label{lem:evo3}
		For all nonempty finite subsets $S\subset V$, $p\in (0,1]$, and $\mu\in (0,1/2]$, if $\bm{\eta}$ satisfies that
		\begin{equation}\label{eq:cond_partialS}
			\left|\{e\in \partial_E S: \eta_t(e) =1 \ \text{for all} \ t\in [\mu^{-1}-1,\mu^{-1}]\}\right|\ge \beta|\partial_E S|
		\end{equation} 
		for some $\beta>0$, then we have that
		\begin{align}\label{eq:PhiSn} \Phi^{\bm{\eta}_{[0,\mu^{-1}]}}_S\ge \frac{\beta}{de}\Phi(G),\quad  \forall n\in \N\cup\{0\}, \ \mu \in (0,1/2]. 
		\end{align}
		Here, $\bm{\eta}_{[0,\mu^{-1}]}$ refers to the entire environment process between times 0 and $\mu^{-1}$, and $\Phi^{\bm{\eta}_{[0,\mu^{-1}]}}_S$ is defined in a similar way as in \eqref{eq.defSn}:
		\begin{equation}
			\label{eq:PhiSnmu}  \Phi_{S}^{\bm{\eta}_{[0,\mu^{-1}]}}=\frac{1}{\left|S\right|} \sum_{x \in S} \sum_{y \in S^{c}} \widehat{\mathbb{P}}^{\bm{\eta}_{[0,\mu^{-1}]}}\left(Y_{1}=y | Y_{0}=x\right).
		\end{equation}
	\end{lemma} 
	\begin{proof}
		We denote
		$$S_{\text{good}}:=\left\{x\in S: \text{there exists an edge $e$ from $x$ to $S^c$ that is open during $[\mu^{-1}-1,\mu^{-1}]$} \right\},$$
		and let $S_{\text{bad}}=S\setminus S_{\text{good}}$. By \eqref{eq:cond_partialS}, we have that
		\begin{equation}\label{eq:Sgood}
			|S_{\text{good}}| \ge \frac{1}{d}\left|\{e\in \partial_E S: \eta_t(e) =1 \ \text{for all} \ t\in [\mu^{-1}-1,\mu^{-1}]\}\right|\ge \frac{\beta}{d}|\partial_E S|\ge \beta\Phi(G)|S| . 
		\end{equation}

		Consider $\Phi_S= \Phi^{\bm{\eta}_{[0,\mu^{-1}]}}_S$ as defined in \eqref{eq:PhiSnmu} and abbreviate $\bm{\eta}= \bm{\eta}_{[0,\mu^{-1}]}$. 
		Since $\pi(x)\equiv 1$ is a stationary measure for all realizations of the environment according to the definition of the random walk, we have  
		\begin{align} \label{eq:trivial_stat}
			\max_{y\in V} \sum_{x \in S} \widehat{\mathbb{P}}^{\bm{\eta}}\left(X_{\mu^{-1}-1}=y | X_{0}=x\right) \le \max_{y\in V} \sum_{x \in V} \widehat{\mathbb{P}}^{\bm{\eta}}\left(X_{\mu^{-1}-1}=y | X_{0}=x\right) \le 1.
		\end{align}
		We observe that 
		\begin{align}\label{eq:PhiS}
			\Phi_S &= \widehat{\mathbb{P}}^{\bm{\eta}}\left(Y_{1}\in S^c | Y_{0}\in S\right) \\
			&   \ge \widehat{\mathbb{P}}^{\bm{\eta}}\left(X_{\mu^{-1}}\in S^c | X_{0}\in S, X_{\mu^{-1}-1}\in S_{\text{good}}\cup S^c\right) \widehat{\mathbb{P}}^{\bm{\eta}}\left(X_{\mu^{-1}-1}\in S_{\text{good}}\cup S^c| X_{0}\in S\right) ,\nonumber
		\end{align}
		where the conditioning $Y_0\in S$ assigns a probability of $|S|^{-1}$ to each point in $S$. Using \eqref{eq:trivial_stat}, we can bound the second factor by 
		\begin{align}
			\widehat{\mathbb{P}}^{\bm{\eta}}\left(X_{\mu^{-1}-1}\in S_{\text{good}}\cup S^c| X_{0}\in S\right) &= 1- \widehat{\mathbb{P}}^{\bm{\eta}}\left(X_{\mu^{-1}-1}\in S_{\text{bad}}| X_{0}\in S\right) \nonumber\\
			&\ge 1-\frac{|S_{\text{bad}}|}{|S|}=\frac{|S_{\text{good}}|}{|S|}\ge \beta\Phi(G) ,\label{eq:Sgood_lower}
		\end{align}
		where we used \eqref{eq:Sgood} in the last step. For the first factor, if $X_{\mu^{-1}-1}\in S_{\text{good}}$, we fix an arbitrary edge $e$ from $X_{\mu^{-1}-1}$ to $S^c$ that is open during $[\mu^{-1}-1,\mu^{-1}]$. The probability that the random walk clock rings exactly once
		during $[\mu^{-1}-1,\mu^{-1}]$ and the attempted jump is along $e$ is at least $(de)^{-1}$. On the other hand, if $X_{\mu^{-1}-1}\in S^c$, then the probability that the random walk clock does not ring during $[\mu^{-1}-1,\mu^{-1}]$ is at least $1-e^{-1}$. Thus, we can bound \eqref{eq:PhiS} by 
		$$ \Phi_S \ge \left( \frac{1}{de}\wedge (1-e^{-1})\right)\beta\Phi(G)= \frac{\beta\Phi(G)}{d e}.$$
		This leads to \eqref{eq:PhiSn}.
	\end{proof}
	
	\begin{corollary}\label{cor:largemu}
		Under the setting of \Cref{prop:general_Lower}, for any $p\in (0,1)$ and $\mu \in (0, 1/2]$, we have 
		\begin{equation}		\label{eq:lowerbound_general_smallmu}		
			v_p(\mu)\ge \mu	\frac{\beta^3\Phi(G)^2}{12e^2d^2\log d},					
		\end{equation}		
		where $\beta=\frac{p}{2}(1-e^{-1/2})e^{-(1-p)/2}$. 
	\end{corollary}
	\begin{proof}
		Lemmas \ref{lem:evo4} and \ref{lem:evo3} together imply that \eqref{eq:PSn_G_lowerbound_musmall}	holds with 
		$$c_0=\frac{\beta}{2},\quad c_1=\beta\Phi(G) .$$
		Following the same proof as in Section~\ref{subsec::DFcoupling}, we obtain an estimate that is similar to \eqref{eqn::super_lowerbound_G} but with the speed scaled by $\mu^{-1}$:
		\[\frac{1}{\mu}v_p(\mu)\ge \frac{c_0c_1^2}{6e^2d^2\log d}= \frac{\beta^3\Phi(G)^2}{12 e^2d^2\log d}.\]
		This gives \eqref{eq:lowerbound_general_smallmu}. 
	\end{proof}

	\begin{proof}[Proof of \Cref{prop:general_Lower}]
		Combining \Cref{cor:smallmu} for the $\mu > 1/2$ case and \Cref{cor:largemu} for the $\mu \le 1/2$ case, we conclude \eqref{eq:lower_vpmu_general}.
	\end{proof}

	\section{Concluding remarks and questions}\label{sec:conclusion}

	We believe that in the critical case, the speed of the random walk should be of order $\mu^\alpha$ for some fixed exponent $\alpha=\alpha(G)>0$ (with potential $\log(1/\mu)$ factors). Our results suggest that $1/2\le \alpha\le 1$. Based on further heuristics, we conjecture that the correct exponent is $\alpha=3/4$ for all nonamenable transitive unimodular graphs. 
	Furthermore, we expect that the diffusion constant $D_{p_c}(\mu)$ for the random walk on critical dynamical percolation in $\z^d$ would exhibit a similar behavior $D_{p_c}(\mu)\sim \mu^{\alpha}$ for a fixed exponent $\alpha=\alpha(\mathbb Z^d)>0$. In the companion paper \cite{GJPSWY_Zd}, we study the random walk on critical dynamical percolation in $\z^d$ and establish a similar bound $1/2\le \alpha(\mathbb Z^d) \le 1$ when $d\ge 11$. As mentioned in the introduction, it is commonly believed that the random walk on critical percolation in high-dimensional $\z^d$ lattices is closely connected to that in trees. Consequently, we conjecture that $\alpha(\mathbb Z^d)$ should match the exponent observed in trees, e.g., $\alpha(\mathbb Z^d)=\alpha(\mathbb T_{b+1})$ for large $d$ (e.g., $d\ge 11$) and any $b\ge 2$.

	The lower bound for the speed of the random walk on supercritical dynamical percolation in trees is quantitive in terms of $d$ and $p$ as shown in \Cref{thm::supercritical_quantitative}. In particular, for the near-critical case when $p\downarrow p_c$, the dependence of the lower bound on $p-p_c$ is shown in \eqref{eq:lowerv9}.  
	We now consider the simple random walk on a Galton-Watson (GW) tree with 
	$$p_k={b\choose k}p^k (1-p)^{b-k},\quad k\in \{0,1,\ldots, b\}. $$
	This walk is closely related to the random walk on Bernoulli-$p$ percolation in $\mathbb{T}_d$. By Theorem 17.13 and Exercise 17.7 of \cite{LP16} (or Theorem 3.2 of \cite{LPP95} by Lyons, Pemantle, and Peres), we know that the speed on the GW tree is
	\begin{equation}\label{eq:wtvp}
		\wt v(p)=\sum_{k=0}^{b}\frac{k-1}{k+1}p_k\frac{1-q^{k+1}}{1-q^2}.
	\end{equation}
	Here, $q=1-\wt\theta_p$ is the extinction probability, satisfying the equation \( (1-p\wt\theta_p)^b=1-\wt\theta_p.\) From this equation, it follows that as $p\downarrow p_c,$
	\begin{equation}\label{eq:asymp_p} p-p_c= \frac{b-1}{2}p^2\wt\theta_p - \frac{(b-1)(b-2)}{6}p^3\wt\theta_p^2 +O(\wt\theta_p^3). 
	\end{equation}
	By performing a Taylor expansion of \eqref{eq:wtvp} around $\wt\theta_p=0$, we obtain that 
	\begin{align}
		\wt v(p)&=\frac{1}{2-\wt\theta_p}\sum_{k=0}^{b} (k-1) p_k \left( 1- \frac{k}{2}\wt\theta_p+\frac{k(k-1)}{6}\wt\theta_p^2\right) +O(\wt\theta_p^3) \nonumber\\
		&=\frac{1}{2-\wt\theta_p}\mathbb E \left[(Z-1) -\frac{\wt\theta_p }{2} Z(Z-1) + \frac{\wt\theta_p^2}{6}Z(Z-1)^2\right] +O(\wt\theta_p^3),\label{eq:Taylorv}
	\end{align}
	where $Z\sim\text{Bin}(b, p)$ is a binomial random variable with parameters $b$ and $p$. Using \eqref{eq:asymp_p} and that 
	$$ \mathbb E Z=bp,\quad \e [Z(Z-1)]=b(b-1)p^2,\quad \e[Z(Z-1)^2]= b(b-1)(b-2)p^3+b(b-1)p^2,$$
	we can simplify \eqref{eq:Taylorv} as 
	\begin{align}
		\wt v(p)    &=\frac{b(b-1)p^2}{6(2-\wt\theta_p)}\wt\theta_p^2+O(\wt\theta_p^3) \gtrsim (p-p_c)^2\ \ \text{ as }\ \ p\downarrow p_c,\label{eq:Taylorv2}
	\end{align}
	where we used the asymptotic \eqref{eq:wtThetap} in the second step. 
	Taking into account the fact that $X_0=o$ has a probability $\theta_p\sim p-p_c$ of lying inside an infinite component, we heuristically expect that the random walk on near-critical dynamical percolation in $\mathbb T_d$ should have a speed
	\[ v_p(\mu)\gtrsim \theta_p \wt v(p) \gtrsim  (p-p_c)^3,\quad \text{ as }p\downarrow p_c.\]
	Thus, the exponent shown in \eqref{eq:lowerv9} is likely not sharp. 

	In contrast to the tree case, the lower bound for the speed on general nonamenable graphs depends on an implicit constant $c_p$ in \Cref{lem:trifucation_general}, which in turn relies on the specific graph $G$. It is interesting to investigate whether the speed has a uniform lower bound that only depends on the Cheeger constant, i.e., there exists a positive function $f$ such that for all $p>p_c$,
	$$\inf_{\mu>0} v(\mu, p, G) \ge f(p,d,\Phi(G))>0.$$
	If such a function exists, it would be desirable to obtain more quantitative estimates for it when $p>p_c$. In particular, we will be interested in determining the exact exponent of $p-p_c$ in $v(\mu,p)$ as $p\downarrow p_c$.

	Finally, an important open question concerns the monotonicity of the speed $v(p,\mu)$ as a function of $\mu$ and $p$ for all transitive graphs. Currently, it is unknown whether the speed exhibits monotonicity with respect to either $\mu$ or $p$, but we conjecture that both forms of monotonicity hold. 
	In particular, if the monotonicity in $\mu$ is valid, then the speed of the random walk on static percolation should always give a lower bound for the speed on dynamical percolation. This open question was also proposed at the conclusion of \cite{peres-stauffer-steif} regarding the random walk on dynamical percolation of $\z^d$ lattices.

	\section*{Acknowledgments}
	Chenlin Gu is supported by the National Key R\&D Program of China (No. 2023YFA1010400) and National Natural Science Foundation of China (12301166).
	Jianping Jiang is supported by National Natural Science Foundation of China (12271284 and 12226001). Hao Wu is supported by Beijing Natural Science Foundation (JQ20001). Fan Yang is supported by the National Key R\&D Program of China (No. 2023YFA1010400). 
	
	\bibliographystyle{abbrv}
	\bibliography{ref.bib}

\end{document}